\documentclass[12pt]{article}

\usepackage{amssymb,amsmath,amscd}
\usepackage{mathrsfs}
\usepackage{amsfonts}
\usepackage{amssymb,amsthm}
\usepackage[arc,all]{xy}
\usepackage[utf8x]{inputenc}
\usepackage{hyperref}

\newcommand{\CC}{{\mathcal C}}

\newcommand{\R}{{\bf R}}

\newcommand{\QQ}{\mathbb{Q}}

\newcommand{\Cp}{\mathbb{C}}

 \DeclareMathOperator{\pt}{pt}

\DeclareMathOperator{\gl}{\mathfrak{gl}}
 
\DeclareMathOperator{\Spec}{Spec}

\DeclareMathOperator{\id}{id} 
\DeclareMathOperator{\Ho}{H}

\newtheorem{thm}{Theorem}[section]
\newtheorem{lem}[thm]{Lemma}
\newtheorem{cor}[thm]{Corollary}
\newtheorem{prop}[thm]{Proposition}

\newtheorem{defn}[thm]{Definition}

\newtheorem{conj}[thm]{Conjecture}
\newtheorem{rem}[thm]{Remark}

\newcommand{\bcen}{\begin{center}}
\newcommand{\ecen}{\end{center}}

\def\dim{{\rm dim}}
\def\Im{{\rm Im}}

\title{Cohomological Hall algebras,  semicanonical bases and Donaldson-Thomas invariants for $2$-dimensional Calabi-Yau categories \\(with an appendix by Ben Davison)}

\author{Jie Ren and Yan Soibelman}

\begin{document}
\maketitle

{\it to Maxim Kontsevich on his 50th birthday}

\tableofcontents

\section{Introduction}
Representations of a quiver $Q$ give rise in general to the category of cohomological dimension one. In this sense quivers are  analogous
to  curves.

In a similar vein representations of the preprojective algebra $\Pi_Q$ give rise  in general to the category of cohomological dimension two. Since  preprojective algebras are obtained as simplectic reductions, we can say that they are analogous to  Higgs bundles
on a curve. Likewise the deformed preprojective algebras $\Pi^{\lambda}$ are analogous to  algebraic vector bundles
with connections on a curve.  The interplay between algebraic and geometric sides of this
dictionary is a subject of many papers (here are few random samples: \cite{CB1}, \cite{Ef}, \cite{S}).

The above categories can be ``upgraded" to $3$-dimensional Calabi-Yau categories. On  the algebraic
side it can be achieved by constructing a ``triple'' quiver
$\widehat{Q}$ endowed with a cubic potential $W$, see e.g. \cite{Moz1} or Section 2 below. On the geometric side one takes the total space of an appropriate rank two vector bundle on a curve or the total space of the anticanonical bundle on a surface.
From the point of view of, say, gauge theory, the lower-dimensional categories can be thought of as ``dimensional reductions''
of the corresponding
$3$-dimensional Calabi-Yau categories.

The framework of $3$-dimensional Calabi-Yau categories is appropriate for the theory
of motivic Donaldson-Thomas invariants (see \cite{KoSo2}, \cite{KoSo3}). The above-mentioned dimensional reduction gives rise to the corresponding theory
in a lower dimension.  It is natural to ask about the meaning of the objects arising  as a result
of such  dimensional reduction.
Our paper illustrates this philosophy in two examples.

In our first example we discuss
semicanonical bases (see \cite {L1}, \cite{L2}) from the point of view of Cohomological Hall algebras (see \cite{KoSo3}).
In our second example we formulate  a conjecture about an analog
of the Kac polynomial of a $2$-dimensional Calabi-Yau category. The conjecture is motivated by \cite{Moz1} in which the motive of the stack of indecomposable representations of a quiver
(Kac polynomial) was expressed in terms of the motives of stacks of representations of the corresponding preprojective algebra and the Donaldson-Thomas invariants of the corresponding $3CY$ category.

It is known after Lusztig (see \cite{L1}, \cite{L2}) that
semicanonical bases can be interpreted in terms of top-dimensional
components of the stack of nilpotent representations of the
preprojective algebra $\Pi_Q$ associated with the quiver $Q$. On the
other hand, the  critical loci of $Tr(W)$ contains the stack of
representations of $\Pi_Q$. This gives an idea that the
semicanonical basis can be derived from the pair $(\widehat{Q},W)$.

Indeed, the pair $(\widehat{Q},W)$ gives rise to the corresponding Cohomological Hall algebra (COHA for short), see \cite{KoSo3}.
It is natural to transport the associative product from the COHA to the vector space generated by the
above-mentioned top-dimensional components (in a more invariant way one can
speak about the Galois-invariant part of COHA, cf. \cite{KoSo3}).
We will show that in this way we obtain  an associative algebra endowed with a basis which enjoys the properties of  the semicanonical one.

From the point of view of the above  dictionary between algebra and geometry, the preprojective algebra $\Pi_Q$ should be replaced in general by a
$2$-dimensional Calabi-Yau category ($2CY$ category for short), while the pair  $(\widehat{Q},W)$ should be replaced  by a $3$-dimensional
Calabi-Yau category ($3CY$ category for short). The ``dimensional reduction" of this  $3CY$
categories to a $2$-dimensional category (as well as the relation
between the corresponding cohomology groups of stacks of objects) was described in  \cite{KoSo3}, Section 4.8. One can expect a similar story in the general categorical framework.

In particular we expect that the notion of  semicanonical basis has intrinsic categorical meaning for  $2CY$ categories.
A class of such categories is proposed in Section 3.2. The categories
from our class are parametrized by quivers.

One can also hope that the reduction
from $3CY$ categories to $2$-dimensional categories will help to understand the relation between motivic Donaldson-Thomas theory
(see \cite{KoSo2}, \cite{KoSo3})
and some invariants of $2CY$ categories, e.g. Kac polynomials.\footnote{The relationship between motivic DT-invariants and Kac polynomials
was studied in many papers, including \cite{D2}, \cite{Moz1}, \cite{HLRV}.}

It is not surprising that some notions introduced for Kac-Moody algebras can be interpreted in terms of $2CY$ categories, since
Kac-Moody algebras arise naturally in the framework of $2CY$ categories generated by spherical collections a'la \cite{KoSo3}.
Spherical objects generate a  $t$-structure of a
$2CY$ category. Its heart corresponds to a Borel subalgebra, and the spherical
objects themselves correspond to simple roots. In a similar vein Schur objects correspond to positive roots, etc.
Change of the $t$-structure corresponds to the change of the cone of positive roots
(and hence to the change of Borel subalgebra in the corresponding Kac-Moody algebra). Reflections at
spherical objects correspond to the action of the Weyl group. Space of stability functions
(i.e. central charges) on a fixed $t$-structure corresponds to the Cartan subalgebra.
It contains the lattice which is dual to the $K$-theory classes of spherical objects.
The Euler bilinear form on the $K$-group of the category
(it is symmetric due to the $2CY$ condition) corresponds to the Killing form on the Cartan subalgebra.
The moduli stack of objects of the heart of the $t$-structure is symplectic, and
it contains a Lagrangian substack (in fact subvariety).
\footnote{This Lagrangian subvariety should be compared with Lusztig's Lagrangian nilpotent cone,
whose irreducible components correspond to positive roots. In the spirit of the analogy with Higgs bundles it is similar to the Laumon nilpotent cone.}

Notice that
every $2CY$ category gives rise to a $3CY$ category with the trivial Euler form. Geometrically, the underlying Calabi-Yau $3$-fold is the product
$S\times {\mathbb A}^1$, where $S$ is a Calabi-Yau surface.
Upgrade from $\Pi_Q$ to $(\widehat{Q},W)$ illustrates the algebraic side of the dictionary (in general the $3CY$ category should have a $t$-structure with the heart containing  pairs $(E,f)$, where $E$ is an object of the heart of a fixed $t$-structure of our $2CY$ category and $ f\in Hom(E,E)$).

Motivic  Donaldson-Thomas invariants of such ``upgraded" $3CY$ categories do not change inside
of a connected component of the space of stability conditions. As a result, the DT-invariants are
in fact invariants of the $t$-structure of the underlying $2CY$ category.

Another interesting question is the relation of the representation theory of COHA as described in \cite{So} to the one of the algebras acting on the cohomologies of the moduli spaces of instantons on ${\mathbb C}{\mathbb P}^2$ (AGT conjecture, etc.). In the Appendix written by Ben Davison the relation between COHA and cohomological Hall algebra of the commuting variety introduced in \cite{SV} is explained (in fact the relation takes place for a general quiver). On the other hand, representations of $W$-algebras used in the loc.cit. for the proof of the AGT conjecture are derived from the
cohomological Hall algebra of the commuting variety. We hope that the appearance of COHA in this story is not accidental and plan to study that topic in the future.

Finally, we remark that there are  $2CY$ categories which have purely geometric origin
(e.g. the category of coherent sheaves on a $K3$ surface).
The above questions about semicanonical basis or Kac polynomial make sense in the geometric case as well.

{\it Acknowledgements}. The idea of the paper goes back to the discussions of the second author (Y.S.)  with  Maxim Kontsevich during the conference on DT-invariants in Budapest in 2012 (the conference was sponsored by AIM). We are grateful to Maxim Kontsevich for sharing with us his ideas on the subject. We also thank to George Lusztig for encouraging us to continue the study of the relationship between
COHA and semicanonical bases.
Y.S.  thanks to IHES for excellent research conditions. His work was partially supported by an NSF grant.

%\chapter{COHA}

\section{COHA and semicanonical basis}

We start by recalling the definition of the product on the critical COHA
proposed in \cite{KoSo3}. For the convenience of the reader we will closely
follow the very detailed exposition from \cite{D1} which contains
proofs of several statements  sketched in \cite{KoSo3} as well
as several useful improvements of the loc.cit.

\subsection{Reminder on the critical cohomology}

Let $Y$ be a complex manifold, and $f:Y\rightarrow\mathbb{C}$ a
holomorphic function. We define the vanishing cycles functor
$\varphi_f$ as follows:
$\varphi_{f}\mathcal{F}[-1]:=(R\Gamma_{\{Re(f)\leq0\}}\mathcal{F})_{f^{-1}(0)}$.
This is a nonstandard definition of this functor, which is
equivalent to the usual one in the complex case. From now on we will
abbreviate $R\mathscr{F}$ to $\mathscr{F}$ for any functor
$\mathscr{F}$. There is an isomorphism
$\mathbb{Q}_{Y}\otimes\mathbb{T}^{\dim Y}\rightarrow
D\mathbb{Q}_{Y}$, where ${\mathbb T}$ is the Tate motive, and $D$ is
the Verdier duality. The above isomorphism induces an isomorphism
\begin{equation}
\varphi_{f}\mathbb{Q}_{Y}\otimes\mathbb{T}^{\dim
Y}\rightarrow\varphi_{f} D\mathbb{Q}_{Y}.
\end{equation}
If $g:Y'\rightarrow Y$ is a map between manifolds, then the natural
transformation of functors $\Gamma_{\{Re(f)\leq0\}}\rightarrow
g_{*}\Gamma_{\{Re(fg)\leq0\}}g^{*}$ induces a natural transformation
\begin{equation}
\varphi_{f}\rightarrow g_{*}\varphi_{fg}g^{*}.
\end{equation}
Assume that $g$ is an affine fibration, then by
\cite[Cor.~2.3]{D1} there is a natural equivalence
\begin{equation}
\varphi_{f}g_{!}g_{*}\rightarrow g_{!}\varphi_{fg}g_{*}.
\end{equation}

\begin{defn}
For any submanifold $Y^{sp}\subset Y$, the critical cohomology with
compact support $H_{c}^{\bullet,crit}(Y^{sp},f)$ is defined as the
cohomology of the following object in
$\mathscr{D}^{b}(\textbf{MMHS})$ ($\textbf{MMHS}$ denotes the
category of monodromic mixed Hodge structures):
$$(\mathbb{C}^{\ast}\rightarrow\mathbb{A}^{1})_{!}(Y^{sp}\times\mathbb{C}^{\ast}\rightarrow\mathbb{C}^{\ast})_{!}(Y^{sp}\times\mathbb{C}^{\ast}\rightarrow Y\times\mathbb{C}^{\ast})^{\ast}\varphi_{\frac{f}{u}}\mathbb{Q}_{Y\times\mathbb{C}^{\ast}},$$
where u is the coordinate on $\mathbb{C}^{\ast}$.
\end{defn}

Let $Y=X\times\mathbb{A}^{n}$ be the total space of the trivial
vector bundle, endowed with the $\mathbb{C}^{*}$-action that acts
trivially on $X$ and with weight one on $\mathbb{A}^{n}$. Let $f:
Y\rightarrow\mathbb{A}^{1}$ be a $\mathbb{C}^{*}$-equivariant
holomorphic function, where $\mathbb{C}^{*}$ acts with weight one on
$\mathbb{A}^{1}$. Then $f=\sum_{k=1}^{k=n}f_{k}x_{k}$, where
$\{x_{k}, k=1,...,n\}$ is a linear coordinate system on
$\mathbb{A}^{n}$, and $f_{k}$ are functions on $X$. Let $Z\subset X$
be the reduced scheme which is the vanishing locus of all functions
$f_{k}$. Then $Z$ is independent of the choice of $x_{k}$. Let $\pi:
Y\rightarrow X$ be the natural projection, and $i: Z\rightarrow X$
be the closed inclusion.

\begin{thm}( see {\rm \cite[Cor.~A.6]{D1}})
There is a natural isomorphism of functors in
$\mathscr{D}^{b}(\textbf{MHM}(X))$:

$$\pi_{!}\varphi_{f}\pi^{*}\stackrel{\sim}{\longrightarrow}\pi_{!}\pi^{*}i_{*}i^{*}.$$

In particular,

$$H_{c}^{\bullet,crit}(Y,f)\simeq H_{c}^{\bullet}(Z\times\mathbb{A}^{n}, \mathbb{Q})\simeq H_{c}^{\bullet}(Z, \mathbb{Q})\otimes\mathbb{T}^{n}.$$
Here $\textbf{MHM}(X)$ denotes the category of mixed Hodge modules on $X$.
\end{thm}

If $Y_{i}=X_{i}\times\mathbb{A}^{n_{i}}$ with equivariant functions
$f_{i}$ satisfy the above conditions for $i=1, 2$, then we have

\begin{thm}(see {\rm \cite[Prop.~A.5]{D1}})

The following diagram of isomorphisms commutes:

$$\begin{xy} (0,20)*+{H_{c}^{\bullet,crit}(Y_{1}\times
Y_{2},f_{1}\boxplus f_{2})}="v1";
(80,20)*+{H_{c}^{\bullet,crit}(Y_{1},f_{1})\otimes
H_{c}^{\bullet,crit}(Y_{2},f_{2})}="v2";
(0,0)*+{H_{c}^{\bullet}(Z_{1}\times
Z_{2}\times\mathbb{A}^{n_{1}+n_{2}},\mathbb{Q})}="v3";
(80,0)*+{H_{c}^{\bullet}(Z_{1}\times\mathbb{A}^{n_{1}},\mathbb{Q})\otimes
H_{c}^{\bullet}(Z_{2}\times\mathbb{A}^{n_{2}},\mathbb{Q})}="v4";
{\ar@{->}^{TS} "v1";"v2"}; {\ar@{->}^{Ku} "v3";"v4"};{\ar@{->}
"v1";"v3"}; {\ar@{->} "v2";"v4"}
\end{xy}$$
Here $TS$ denotes the Thom-Sebastiani isomorphism, and $Ku$ the
K\"unneth isomorphism (see loc.cit.).
\end{thm}

\begin{cor}(see {\rm \cite[Cor.~A.7]{D1}})

Let $X^{sp}\subset X$ be a subvariety of $X$ and
$Y^{sp}=X^{sp}\times\mathbb{A}^{n}$, $Z^{sp}=Z\cap X^{sp}$. There is
a natural isomorphism in \textbf{MMHS}

$$H_{c}^{\bullet,crit}(Y^{sp},f)\simeq H_{c}^{\bullet}(Z^{sp}\times\mathbb{A}^{n}, \mathbb{Q}).$$
\end{cor}

The above statements also hold in equivariant case. Let us recall
that framework. Assume that $Y$ is a $G$-equivariant vector bundle
over X, where $G$ is an algebraic group embedded in
$GL(n,\mathbb{C})$, and $f: Y\rightarrow\mathbb{A}^{1}$ is
$G$-invariant. Let $fr(n,N)$ be the space of $n$-tuples of linearly
independent vectors in $\mathbb{C}^{N}$ for $N\geqslant n$, and
$\overline{(Y,G)}_{N}:=Y\times_{G}fr(n,N)$. We denote the induced
function by $f_{N}: \overline{(Y,G)}_{N}\rightarrow\mathbb{A}^{1}$.
For  a $G$-invariant closed subset  $Y'\subset Y$, we define the
equivariant cohomology with compact support by
$H_{c,G}^{\bullet,crit}(Y',f):=\lim\limits_{N\rightarrow\infty}H_{c}^{\bullet,crit}(Y'_{N},f_{N})\otimes\mathbb{T}^{-dim(fr(n,N))}$,
where $Y'_{N}\subset\overline{(Y,G)}_{N}$ is the subspace of points
projected to $Y'$.

\begin{thm} (see {\rm \cite[Cor.~A.8]{D1}})
Let $Y^{sp}=X^{sp}\times\mathbb{A}^{n}$ be the total space
of a sub G-bundle. Then there is an isomorphism in \textbf{MMHS}

$$H_{c,G}^{\bullet,crit}(Y^{sp},f)\simeq H_{c,G}^{\bullet}(Z^{sp}\times\mathbb{A}^{n}, \mathbb{Q}).$$

Moreover, the following diagram of isomorphisms commutes:

$$\begin{xy} (0,20)*+{H_{c,G}^{\bullet,crit}(Y_{1}^{sp}\times
Y_{2}^{sp},f_{1}\boxplus f_{2})}="v1";
(80,20)*+{H_{c,G}^{\bullet,crit}(Y_{1}^{sp},f_{1})\otimes
H_{c,G}^{\bullet,crit}(Y_{2}^{sp},f_{2})}="v2";
(0,0)*+{H_{c,G}^{\bullet}(Z_{1}^{sp}\times
Z_{2}^{sp}\times\mathbb{A}^{n_{1}+n_{2}},\mathbb{Q})}="v3";
(80,0)*+{H_{c,G}^{\bullet}(Z_{1}^{sp}\times\mathbb{A}^{n_{1}},\mathbb{Q})\otimes
H_{c,G}^{\bullet}(Z_{2}^{sp}\times\mathbb{A}^{n_{2}},\mathbb{Q})}="v4";
{\ar@{->}^{TS} "v1";"v2"}; {\ar@{->}^{Ku} "v3";"v4"};{\ar@{->}
"v1";"v3"}; {\ar@{->} "v2";"v4"}
\end{xy}$$
\end{thm}

\subsection{COHA  and preprojective algebras}

Let $Q$ be a quiver with the set of vertices $I$ and the set of arrows $\Omega$. We denote by
$a_{ij}\in\mathbb{Z}_{\geqslant0}$ the number of arrows from $i$ to
$j$ for $i, j\in I$.

One constructs the double quiver $\overline{Q}$, the preprojective
algebra $\Pi_{Q}$, and the triple quiver with potential
$(\widehat{Q}, W)$ as follows. For any arrow $a:i\rightarrow
j\in\Omega$, we add an inverse arrow $a^{*}:j\rightarrow i$ to $Q$
to get $\overline{Q}$, then
$\Pi_{Q}=\mathbb{C}\overline{Q}/\sum_{a\in\Omega}[a, a^{*}]$. Adding
loops $l_{i}: i\rightarrow i$ at each vertex $i\in I$ to
$\overline{Q}$ gives us the ``triple" quiver $\widehat{Q}$. It is
endowed with cubic  potential $W=\sum_{a\in\Omega}[a, a^{*}]l$,
where $l=\sum_{i\in I}l_{i}$. For any dimension vector
$\gamma=(\gamma^{i})_{i\in I}\in \mathbb{Z}^{I}_{\geqslant0}$ we
have the following algebraic varieties:

a) the space $\textbf{M}_{\overline{Q},\gamma}$ of representations
of the double quiver $\overline{Q}$  in coordinate spaces
$({\mathbb C}^{\gamma^i})_{i\in I}$ ;

b) the similar space of representations $\textbf{M}_{\Pi_{Q},\gamma}$ of
$\Pi_{Q}$;

c) the similar space of representations $\textbf{M}_{\widehat{Q},\gamma}$  of
$\widehat{Q}$.

All these spaces of representations are endowed with the
action by conjugation of the complex algebraic group
$\textbf{G}_{\gamma}=\prod_{i\in I}GL(\gamma^{i}, \mathbb{C})$.

Let
$\chi_{Q}(\gamma_{1},\gamma_{2})=\chi(Ext^{\bullet}(x_{1},x_{2}))=-\sum\limits_{i,j\in
I}a_{ij}\gamma_{1}^{i}\gamma_{2}^{j}+\sum\limits_{i\in
I}\gamma_{1}^{i}\gamma_{2}^{i}$ be the Euler form on the $K_{0}$
group of the category of finite dimensional representations of $Q$,
where $x_{1}$ and $x_{2}$ are arbitrary representations of $Q$ of
dimension vectors $\gamma_{1}$ and $\gamma_{2}$ respectively.

In the context of the previous section, let
$X=\textbf{M}_{\overline{Q},\gamma}$,
$Y=\textbf{M}_{\widehat{Q},\gamma}=\textbf{M}_{\overline{Q},\gamma}\times\mathbb{A}^{\gamma\cdot\gamma}$
(dot denotes the inner product), and
$f=Tr(W)_{\gamma}=\sum\limits_{i\in
I,k=1,\ldots,(\gamma^{i})^{2}}f_{ik}x_{ik}$, where $f_{ik}$ are
functions on $\textbf{M}_{\overline{Q},\gamma}$, and $\{x_{ik}\}$ is
a linear coordinate system on $\mathbb{A}^{\gamma\cdot\gamma}$. Then
$Z=\textbf{M}_{\Pi_{Q},\gamma}$. Denote by
$\textbf{M}_{\Pi_{Q},\gamma_{1},\gamma_{2}}$ the space of
representations of $\overline{Q}$ in coordinate spaces of dimension
$\gamma_{1}+\gamma_{2}$ such that the standard coordinate subspaces
of dimension $\gamma_{1}$ form a subrepresentation, and the
restriction of $\rho\in\textbf{M}_{\Pi_{Q},\gamma_{1},\gamma_{2}}$
on the block-diagonal part is an element in
$\textbf{M}_{\Pi_{Q},\gamma_1}\times\textbf{M}_{\Pi_{Q},\gamma_2}$.
The group
$\textbf{G}_{\gamma_{1},\gamma_{2}}\subset\textbf{G}_{\gamma}$
consisting of transformations preserving subspaces
$(\mathbb{C}^{\gamma_{1}^{i}}\subset\mathbb{C}^{\gamma^{i}})_{i\in
I}$ acts on $\textbf{M}_{\Pi_{Q},\gamma_{1},\gamma_{2}}$. Suppose
that we are given a collection of $\textbf{G}_{\gamma}$-invariant
closed subsets
$\textbf{M}_{\overline{Q},\gamma}^{sp}\subset\textbf{M}_{\overline{Q},\gamma}$
satisfying the following condition: for any short exact sequence
$0\rightarrow E_{1}\rightarrow E\rightarrow E_{2}\rightarrow0$ of
representations of $\overline{Q}$ with dimension vectors
$\gamma_{1}, \gamma:=\gamma_{1}+\gamma_{2}, \gamma_{2}$
respectively, $E\in\textbf{M}_{\overline{Q},\gamma}^{sp}$ if and
only if $E_{1}\in\textbf{M}_{\overline{Q},\gamma_{1}}^{sp}$, and
$E_{2}\in\textbf{M}_{\overline{Q},\gamma_{2}}^{sp}$. Then
$\textbf{M}_{\widehat{Q},\gamma}^{sp}=\textbf{M}_{\overline{Q},\gamma}^{sp}\times\mathbb{A}^{\gamma\cdot\gamma}$.

The COHA of $(\widehat{Q},W)$ (see {\cite{KoSo3}}) induces the
coproduct on the vector space
$\bigoplus\limits_{\gamma\in\mathbb{Z}^{I}_{\geqslant0}}H_{c,\textbf{G}_{\gamma}}^{\bullet}(\textbf{M}_{\Pi_{Q},\gamma}^{sp},\mathbb{Q})$
as follows:

\begin{itemize}
\item[$\bullet$]$H_{c,\textbf{G}_{\gamma}}^{\bullet}(\textbf{M}_{\Pi_{Q},\gamma}^{sp},\mathbb{Q})\rightarrow
H_{c,\textbf{G}_{\gamma_{1},\gamma_{2}}}^{\bullet}(\textbf{M}_{\Pi_{Q},\gamma}^{sp},\mathbb{Q})$,
which is the pullback associated to the closed embedding of groups
$\textbf{G}_{\gamma_{1},\gamma_{2}}\rightarrow\textbf{G}_{\gamma}$
with proper quotient. The projections $pr_{\gamma_{1},\gamma_{2},N}:
\overline{(\textbf{M}_{\widehat{Q},\gamma},\textbf{G}_{\gamma_{1},\gamma_{2}})}_{N}\rightarrow\overline{(\textbf{M}_{\widehat{Q},\gamma},\textbf{G}_{\gamma})}_{N}$
induce natural transformations of functors
$\varphi_{\gamma,N}\rightarrow(pr_{\gamma_{1},\gamma_{2},N})_{!}\varphi_{\gamma,\gamma_{1},\gamma_{2},N}(pr_{\gamma_{1},\gamma_{2},N})^{*}$
by (2) and properness of $pr_{\gamma_{1},\gamma_{2},N}$, thus give
us
$$(\pi_{\gamma,N})_{!}\varphi_{\gamma/u,N}(\pi_{\gamma,N})^{*}[-1]\rightarrow(\pi_{\gamma,N})_{!}(pr_{\gamma_{1},\gamma_{2},N})_{!}\varphi_{(\gamma,\gamma_{1},\gamma_{2})/u,N}(pr_{\gamma_{1},\gamma_{2},N})^{*}(\pi_{\gamma,N})^{*}[-1].$$

Here $\varphi_{\gamma,N}=\varphi_{Tr(W)_{\gamma,N}}$ is the
vanishing cycles functor of the function $tr(W)_{\gamma,N}$ on
$\overline{(\textbf{M}_{\widehat{Q},\gamma},\textbf{G}_{\gamma})}_{N}$,
and $\varphi_{\gamma,\gamma_{1},\gamma_{2},N}$ corresponds to
$Tr(W)_{\gamma,\gamma_{1},\gamma_{2},N}$ on
$\overline{(\textbf{M}_{\widehat{Q},\gamma},\textbf{G}_{\gamma_{1},\gamma_{2}})}_{N}$.
(Note that in subscript of
$\varphi_{\gamma,\gamma_{1},\gamma_{2},N}$, $\gamma$ indicates the
dimension vector of $\textbf{M}_{\widehat{Q},\gamma}$, and
$\gamma_{1},\gamma_{2}$ indicate those of
$\textbf{G}_{\gamma_{1},\gamma_{2}}$. We will use similar notations
in the subsequent steps.)

Since the following diagram commutes: $$\begin{xy}
(0,20)*+{\overline{(\textbf{M}_{\widehat{Q},\gamma},\textbf{G}_{\gamma_{1},\gamma_{2}})}_{N}\times\mathbb{C}^{*}}="v1";
(80,20)*+{\overline{(\textbf{M}_{\widehat{Q},\gamma},\textbf{G}_{\gamma})}_{N}\times\mathbb{C}^{*}}="v2";
(0,0)*+{\overline{(\textbf{M}_{\overline{Q},\gamma},\textbf{G}_{\gamma_{1},\gamma_{2}})}_{N}\times\mathbb{C}^{*}}="v3";
(80,0)*+{\overline{(\textbf{M}_{\overline{Q},\gamma},\textbf{G}_{\gamma})}_{N}\times\mathbb{C}^{*}}="v4";
{\ar@{->}^{pr_{\gamma_{1},\gamma_{2},N}} "v1";"v2"};
{\ar@{->}^{pr_{\overline{Q},\gamma_{1},\gamma_{2},N}}
"v3";"v4"};{\ar@{->}^{\pi_{\gamma,\gamma_{1},\gamma_{2},N}}
"v1";"v3"}; {\ar@{->}^{\pi_{\gamma,N}} "v2";"v4"}
\end{xy}$$
we have
\begin{equation}
\begin{array}{ll}
&(\pi_{\gamma,N})_{!}(pr_{\gamma_{1},\gamma_{2},N})_{!}\varphi_{(\gamma,\gamma_{1},\gamma_{2})/u,N}(pr_{\gamma_{1},\gamma_{2},N})^{*}(\pi_{\gamma,N})^{*}[-1]
\\\simeq&(pr_{\overline{Q},\gamma_{1},\gamma_{2},N})_{!}(\pi_{\gamma,\gamma_{1},\gamma_{2},N})_{!}\varphi_{(\gamma,\gamma_{1},\gamma_{2})/u,N}(\pi_{\gamma,\gamma_{1},\gamma_{2},N})^{*}(pr_{\overline{Q},\gamma_{1},\gamma_{2},N})^{*}[-1].
\end{array}\nonumber
\end{equation}
By Theorem 2.2, we have two isomorphisms:
$$(\pi_{\gamma,N})_{!}\varphi_{\gamma/u,N}(\pi_{\gamma,N})^{*}[-1]\simeq(\pi_{\gamma,N})_{!}(\pi_{\gamma,N})^{*}(i_{\gamma,N})_{*}(i_{\gamma,N})^{*}$$
and
\begin{equation}
\begin{array}{ll}
&(pr_{\overline{Q},\gamma_{1},\gamma_{2},N})_{!}(\pi_{\gamma,\gamma_{1},\gamma_{2},N})_{!}\varphi_{(\gamma,\gamma_{1},\gamma_{2})/u,N}(\pi_{\gamma,\gamma_{1},\gamma_{2},N})^{*}(pr_{\overline{Q},\gamma_{1},\gamma_{2},N})^{*}[-1]
\\\simeq&(pr_{\overline{Q},\gamma_{1},\gamma_{2},N})_{!}(\pi_{\gamma,\gamma_{1},\gamma_{2},N})_{!}(\pi_{\gamma,\gamma_{1},\gamma_{2},N})^{*}(i_{\gamma,\gamma_{1},\gamma_{2},N})_{*}(i_{\gamma,\gamma_{1},\gamma_{2},N})^{*}(pr_{\overline{Q},\gamma_{1},\gamma_{2},N})^{*}.
\end{array}\nonumber
\end{equation}

Here $i_{\gamma,N}$ and $i_{\gamma,\gamma_{1},\gamma_{2},N}$ are
inclusions, and the subscripts have the same meaning as the
vanishing cycles functors above.

Pulling back to
$\textbf{M}_{\overline{Q},\gamma,N}^{sp}\times\mathbb{C}^{*}$ gives
us the commutative diagram
$$\begin{xy}
(0,20)*+{H_{c,\textbf{G}_{\gamma}}^{\bullet,crit}(\textbf{M}_{\widehat{Q},\gamma}^{sp},W_{\gamma})}="v1";
(60,20)*+{H_{c,\textbf{G}_{\gamma_{1},\gamma_{2}}}^{\bullet,crit}(\textbf{M}_{\widehat{Q},\gamma}^{sp},W_{\gamma})}="v2";
(0,0)*+{H_{c,\textbf{G}_{\gamma}}^{\bullet}(\textbf{M}_{\Pi_{Q},\gamma}^{sp},\mathbb{Q})\otimes\mathbb{T}^{\gamma\cdot\gamma}}="v3";
(60,0)*+{H_{c,\textbf{G}_{\gamma_{1},\gamma_{2}}}^{\bullet}(\textbf{M}_{\Pi_{Q},\gamma}^{sp},\mathbb{Q})\otimes\mathbb{T}^{\gamma\cdot\gamma}}="v4";
{\ar@{->} "v1";"v2"}; {\ar@{->} "v3";"v4"}; {\ar@{->}^{\wr}
"v1";"v3"}; {\ar@{->}^{\wr} "v2";"v4"}
\end{xy}$$

\item[$\bullet$]$H_{c,\textbf{G}_{\gamma_{1},\gamma_{2}}}^{\bullet}(\textbf{M}_{\Pi_{Q},\gamma}^{sp},\mathbb{Q})\rightarrow
H_{c,\textbf{G}_{\gamma_{1},\gamma_{2}}}^{\bullet}(\widetilde{\textbf{M}}_{\Pi_{Q},\gamma_{1},\gamma_{2}}^{sp},\mathbb{Q})\otimes\mathbb{T}^{-\gamma_{1}\cdot\gamma_{2}}$,
where
$\textbf{M}_{\Pi_{Q},\gamma_{1},\gamma_{2}}^{sp}\\=\textbf{M}_{\Pi_{Q},\gamma}^{sp}\cap\textbf{M}_{\Pi_{Q},\gamma_{1},\gamma_{2}}$,
and
$\widetilde{\textbf{M}}_{\Pi_{Q},\gamma_{1},\gamma_{2}}^{sp}\subset\textbf{M}_{\Pi_{Q},\gamma_{1},\gamma_{2}}$
is the pullback of
$\textbf{M}_{\Pi_{Q},\gamma_{1}}^{sp}\times\textbf{M}_{\Pi_{Q},\gamma_{2}}^{sp}$
under the projection
$\textbf{M}_{\Pi_{Q},\gamma_{1},\gamma_{2}}\rightarrow\textbf{M}_{\Pi_{Q},\gamma_{1}}\times\textbf{M}_{\Pi_{Q},\gamma_{2}}$.
This is the pullback associated to the closed embedding
$\textbf{M}_{\Pi_{Q},\gamma_{1},\gamma_{2}}\rightarrow\textbf{M}_{\Pi_{Q},\gamma}$.
The inclusions $j_{\gamma_{1},\gamma_{2},N}:
\overline{(\textbf{M}_{\widehat{Q},\gamma_{1},\gamma_{2}},\textbf{G}_{\gamma_{1},\gamma_{2}})}_{N}\rightarrow\overline{(\textbf{M}_{\widehat{Q},\gamma},\textbf{G}_{\gamma_{1},\gamma_{2}})}_{N}$
induce natural transformations of functors
$\varphi_{\gamma,\gamma_{1},\gamma_{2},N}\rightarrow(j_{\gamma_{1},\gamma_{2},N})_{*}\varphi_{\gamma_{1},\gamma_{2},N}(j_{\gamma_{1},\gamma_{2},N})^{*}$
by (2). So we have
\begin{equation}
\begin{array}{ll}
&(\pi_{\gamma,\gamma_{1},\gamma_{2},N})_{!}\varphi_{(\gamma,\gamma_{1},\gamma_{2})/u,N}(\pi_{\gamma,\gamma_{1},\gamma_{2},N})^{*}[-1]
\\\rightarrow&(\pi_{\gamma,\gamma_{1},\gamma_{2},N})_{!}(j_{\gamma_{1},\gamma_{2},N})_{*}\varphi_{(\gamma_{1},\gamma_{2})/u,N}(j_{\gamma_{1},\gamma_{2},N})^{*}(\pi_{\gamma,\gamma_{1},\gamma_{2},N})^{*}[-1].
\end{array}\nonumber
\end{equation}
By the commutative diagram $$\begin{xy}
(0,20)*+{\overline{(\textbf{M}_{\widehat{Q},\gamma_{1},\gamma_{2}},\textbf{G}_{\gamma_{1},\gamma_{2}})}_{N}\times\mathbb{C}^{*}}="v1";
(80,20)*+{\overline{(\textbf{M}_{\widehat{Q},\gamma},\textbf{G}_{\gamma_{1},\gamma_{2}})}_{N}\times\mathbb{C}^{*}}="v2";
(0,0)*+{\overline{(\textbf{M}_{\overline{Q},\gamma_{1},\gamma_{2}},\textbf{G}_{\gamma_{1},\gamma_{2}})}_{N}\times\mathbb{C}^{*}}="v3";
(80,0)*+{\overline{(\textbf{M}_{\overline{Q},\gamma},\textbf{G}_{\gamma_{1},\gamma_{2}})}_{N}\times\mathbb{C}^{*}}="v4";
{\ar@{->}^{j_{\gamma_{1},\gamma_{2},N}} "v1";"v2"};
{\ar@{->}^{j_{\overline{Q},\gamma_{1},\gamma_{2},N}}
"v3";"v4"};{\ar@{->}^{\pi_{\gamma_{1},\gamma_{2},N}} "v1";"v3"};
{\ar@{->}^{\pi_{\gamma,\gamma_{1},\gamma_{2},N}} "v2";"v4"}
\end{xy}$$
we have
\begin{equation}
\begin{array}{ll}
&(\pi_{\gamma,\gamma_{1},\gamma_{2},N})_{!}(j_{\gamma_{1},\gamma_{2},N})_{*}\varphi_{(\gamma_{1},\gamma_{2})/u,N}(j_{\gamma_{1},\gamma_{2},N})^{*}(\pi_{\gamma,\gamma_{1},\gamma_{2},N})^{*}[-1]
\\\simeq&(j_{\overline{Q},\gamma_{1},\gamma_{2},N})_{*}(\pi_{\gamma_{1},\gamma_{2},N})_{!}\varphi_{(\gamma_{1},\gamma_{2})/u,N}(\pi_{\gamma_{1},\gamma_{2},N})^{*}(j_{\overline{Q},\gamma_{1},\gamma_{2},N})^{*}[-1].
\end{array}\nonumber
\end{equation}
Then the isomorphisms
\begin{equation}
\begin{array}{ll}
&(\pi_{\gamma,\gamma_{1},\gamma_{2},N})_{!}\varphi_{(\gamma,\gamma_{1},\gamma_{2})/u,N}(\pi_{\gamma,\gamma_{1},\gamma_{2},N})^{*}[-1]
\\\simeq&(\pi_{\gamma,\gamma_{1},\gamma_{2},N})_{!}(\pi_{\gamma,\gamma_{1},\gamma_{2},N})^{*}(i_{\gamma,\gamma_{1},\gamma_{2},N})_{*}(i_{\gamma,\gamma_{1},\gamma_{2},N})^{*}
\end{array}\nonumber
\end{equation}
and
\begin{equation}
\begin{array}{ll}
&(j_{\overline{Q},\gamma_{1},\gamma_{2},N})_{*}(\pi_{\gamma_{1},\gamma_{2},N})_{!}\varphi_{(\gamma_{1},\gamma_{2})/u,N}(\pi_{\gamma_{1},\gamma_{2},N})^{*}(j_{\overline{Q},\gamma_{1},\gamma_{2},N})^{*}[-1]
\\\simeq&(j_{\overline{Q},\gamma_{1},\gamma_{2},N})_{*}(\pi_{\gamma_{1},\gamma_{2},N})_{!}(\pi_{\gamma_{1},\gamma_{2},N})^{*}(i_{\gamma_{1},\gamma_{2},N})_{*}(i_{\gamma_{1},\gamma_{2},N})^{*}(j_{\overline{Q},\gamma_{1},\gamma_{2},N})^{*}
\end{array}\nonumber
\end{equation}
obtained from the theorem give us the commutative diagram by pulling
back to
$\textbf{M}_{\overline{Q},\gamma,\gamma_{1},\gamma_{2},N}^{sp}\times\mathbb{C}^{*}$:
$$\begin{xy}
(0,20)*+{H_{c,\textbf{G}_{\gamma_{1},\gamma_{2}}}^{\bullet,crit}(\textbf{M}_{\widehat{Q},\gamma}^{sp},W_{\gamma})}="v1";
(100,20)*+{H_{c,\textbf{G}_{\gamma_{1},\gamma_{2}}}^{\bullet,crit}(\widetilde{\textbf{M}}_{\widehat{Q},\gamma_{1},\gamma_{2}}^{sp},W_{\gamma_{1},\gamma_{2}})}="v2";
(0,0)*+{H_{c,\textbf{G}_{\gamma_{1},\gamma_{2}}}^{\bullet}(\textbf{M}_{\Pi_{Q},\gamma}^{sp},\mathbb{Q})\otimes\mathbb{T}^{\gamma\cdot\gamma}}="v3";
(100,0)*+{H_{c,\textbf{G}_{\gamma_{1},\gamma_{2}}}^{\bullet}(\widetilde{\textbf{M}}_{\Pi_{Q},\gamma_{1},\gamma_{2}}^{sp},\mathbb{Q})\otimes\mathbb{T}^{l_{1}}}="v4";
(20,40)*+{H_{c,\textbf{G}_{\gamma_{1},\gamma_{2}}}^{\bullet,crit}(\textbf{M}_{\widehat{Q},\gamma_{1},\gamma_{2}}^{sp},W_{\gamma})}="v5";
(80,40)*{H_{c,\textbf{G}_{\gamma_{1},\gamma_{2}}}^{\bullet,crit}(\widetilde{\textbf{M}}_{\widehat{Q},\gamma_{1},\gamma_{2}}^{sp},W_{\gamma})}="v6";
{\ar@{->} "v3";"v4"}; {\ar@{->}^{\wr} "v1";"v3"}; {\ar@{->}^{\wr}
"v2";"v4"}; {\ar@{->} "v1";"v5"}; {\ar@{->}^{\thicksim} "v5";"v6"};
{\ar@{->} "v6";"v2"}
\end{xy}$$

where $l_{1}=\gamma\cdot\gamma-\gamma_{1}\cdot\gamma_{2}$.

\item[$\bullet$]$H_{c,\textbf{G}_{\gamma_{1},\gamma_{2}}}^{\bullet}(\widetilde{\textbf{M}}_{\Pi_{Q},\gamma_{1},\gamma_{2}}^{sp},\mathbb{Q})\stackrel{\sim}{\longrightarrow}
H_{c,\textbf{G}_{\gamma_{1}}\times
\textbf{G}_{\gamma_{2}}}^{\bullet}(\widetilde{\textbf{M}}_{\Pi_{Q},\gamma_{1},\gamma_{2}}^{sp},\mathbb{Q})\otimes\mathbb{T}^{-\gamma_{1}\cdot\gamma_{2}}$.
The affine fibrations $q_{\gamma_{1},\gamma_{2},N}:
\overline{(\textbf{M}_{\widehat{Q},\gamma_{1},\gamma_{2}},\textbf{G}_{\gamma_{1}}\times
\textbf{G}_{\gamma_{2}})}_{N}\rightarrow\overline{(\textbf{M}_{\widehat{Q},\gamma_{1},\gamma_{2}},\textbf{G}_{\gamma_{1},\gamma_{2}})}_{N}$
induce isomorphisms
$\varphi_{(\gamma_{1},\gamma_{2})/u,N}(\mathbb{Q}_{\overline{(\textbf{M}_{\widehat{Q},\gamma_{1},\gamma_{2}},\textbf{G}_{\gamma_{1},\gamma_{2}})}_{N}\times\mathbb{C}^{*}}
\\\stackrel{\sim}{\longrightarrow}(q_{\gamma_{1},\gamma_{2},N})_{*}\mathbb{Q}_{\overline{(\textbf{M}_{\widehat{Q},\gamma_{1},\gamma_{2}},\textbf{G}_{\gamma_{1}}\times
\textbf{G}_{\gamma_{2}})}_{N}\times\mathbb{C}^{*}})$. By applying
Verdier duality we get
$\varphi_{(\gamma_{1},\gamma_{2})/u,N}((q_{\gamma_{1},\gamma_{2},N})_{!}D\mathbb{Q}_{\overline{(\textbf{M}_{\widehat{Q},\gamma_{1},\gamma_{2}},\textbf{G}_{\gamma_{1}}\times
\textbf{G}_{\gamma_{2}})}_{N}\times\mathbb{C}^{*}}\stackrel{\sim}{\longrightarrow}D\mathbb{Q}_{\overline{(\textbf{M}_{\widehat{Q},\gamma_{1},\gamma_{2}},\textbf{G}_{\gamma_{1},\gamma_{2}})}_{N}\times\mathbb{C}^{*}})$.
Then
$\varphi_{(\gamma_{1},\gamma_{2})/u,N}((q_{\gamma_{1},\gamma_{2},N})_{!}\mathbb{Q}_{\overline{(\textbf{M}_{\widehat{Q},\gamma_{1},\gamma_{2}},\textbf{G}_{\gamma_{1}}\times
\textbf{G}_{\gamma_{2}})}_{N}\times\mathbb{C}^{*}}\stackrel{\sim}{\longrightarrow}\mathbb{Q}_{\overline{(\textbf{M}_{\widehat{Q},\gamma_{1},\gamma_{2}},\textbf{G}_{\gamma_{1},\gamma_{2}})}_{N}\times\mathbb{C}^{*}}\otimes\mathbb{T}^{\gamma_{1}\cdot\gamma_{2}})$
by (1), and
$(q_{\gamma_{1},\gamma_{2},N})_{!}\varphi_{(\gamma_{1},\gamma_{2},\gamma_{1}\times\gamma_{2})/u,N}\mathbb{Q}_{\overline{(\textbf{M}_{\widehat{Q},\gamma_{1},\gamma_{2}},\textbf{G}_{\gamma_{1}}\times
\textbf{G}_{\gamma_{2}})}_{N}\times\mathbb{C}^{*}}\\\simeq(q_{\gamma_{1},\gamma_{2},N})_{!}\varphi_{(\gamma_{1},\gamma_{2},\gamma_{1}\times\gamma_{2})/u,N}(q_{\gamma_{1},\gamma_{2},N})^{*}\mathbb{Q}_{\overline{(\textbf{M}_{\widehat{Q},\gamma_{1},\gamma_{2}},\textbf{G}_{\gamma_{1},\gamma_{2}})}_{N}\times\mathbb{C}^{*}}\\\stackrel{\sim}{\longrightarrow}\varphi_{(\gamma_{1},\gamma_{2})/u,N}(\mathbb{Q}_{\overline{(\textbf{M}_{\widehat{Q},\gamma_{1},\gamma_{2}},\textbf{G}_{\gamma_{1},\gamma_{2}})}_{N}\times\mathbb{C}^{*}}\otimes\mathbb{T}^{\gamma_{1}\cdot\gamma_{2}})$
by (3). Then we have isomorphisms
\begin{equation}
\begin{array}{ll}
&(\pi_{\gamma_{1},\gamma_{2},N})_{!}\varphi_{(\gamma_{1},\gamma_{2})/u,N}(\pi_{\gamma_{1},\gamma_{2},N})^{*}(\mathbb{Q}_{\overline{(\textbf{M}_{\overline{Q},\gamma_{1},\gamma_{2}},\textbf{G}_{\gamma_{1},\gamma_{2}})}_{N}\times\mathbb{C}^{*}}\otimes\mathbb{T}^{\gamma_{1}\cdot\gamma_{2}})
\\\rightarrow&(\pi_{\gamma_{1},\gamma_{2},N})_{!}(q_{\gamma_{1},\gamma_{2},N})_{!}\varphi_{(\gamma_{1},\gamma_{2},\gamma_{1}\times\gamma_{2})/u,N}(q_{\gamma_{1},\gamma_{2},N})^{*}(\pi_{\gamma_{1},\gamma_{2},N})^{*}\mathbb{Q}_{\overline{(\textbf{M}_{\overline{Q},\gamma_{1},\gamma_{2}},\textbf{G}_{\gamma_{1},\gamma_{2}})}_{N}\times\mathbb{C}^{*}}.
\end{array}\nonumber
\end{equation}
The commutative diagram $$\begin{xy}
(0,20)*+{\overline{(\textbf{M}_{\widehat{Q},\gamma_{1},\gamma_{2}},\textbf{G}_{\gamma_{1}}\times
\textbf{G}_{\gamma_{2}})}_{N}\times\mathbb{C}^{*}}="v1";
(80,20)*+{\overline{(\textbf{M}_{\widehat{Q},\gamma_{1},\gamma_{2}},\textbf{G}_{\gamma_{1},\gamma_{2}})}_{N}\times\mathbb{C}^{*}}="v2";
(0,0)*+{\overline{(\textbf{M}_{\overline{Q},\gamma_{1},\gamma_{2}},\textbf{G}_{\gamma_{1}}\times
\textbf{G}_{\gamma_{2}})}_{N}\times\mathbb{C}^{*}}="v3";
(80,0)*+{\overline{(\textbf{M}_{\overline{Q},\gamma_{1},\gamma_{2}},\textbf{G}_{\gamma_{1},\gamma_{2}})}_{N}\times\mathbb{C}^{*}}="v4";
{\ar@{->}^{q_{\gamma_{1},\gamma_{2},N}} "v1";"v2"};
{\ar@{->}^{q_{\overline{Q},\gamma_{1},\gamma_{2},N}}
"v3";"v4"};{\ar@{->}^{\pi_{\gamma_{1},\gamma_{2},\gamma_{1}\times\gamma_{2},N}}
"v1";"v3"}; {\ar@{->}^{\pi_{\gamma_{1},\gamma_{2},N}} "v2";"v4"}
\end{xy}$$
gives us isomorphisms
\begin{equation}
\begin{array}{ll}
&(\pi_{\gamma_{1},\gamma_{2},N})_{!}(q_{\gamma_{1},\gamma_{2},N})_{!}\varphi_{(\gamma_{1},\gamma_{2},\gamma_{1}\times\gamma_{2})/u,N}(q_{\gamma_{1},\gamma_{2},N})^{*}(\pi_{\gamma_{1},\gamma_{2},N})^{*}[-1]
\\\simeq&(q_{\overline{Q},\gamma_{1},\gamma_{2},N})_{!}(\pi_{\gamma_{1},\gamma_{2},\gamma_{1}\times\gamma_{2},N})_{!}\varphi_{(\gamma_{1},\gamma_{2},\gamma_{1}\times\gamma_{2})/u,N}(\pi_{\gamma_{1},\gamma_{2},\gamma_{1}\times\gamma_{2},N})^{*}(q_{\overline{Q},\gamma_{1},\gamma_{2},N})^{*}[-1].
\end{array}\nonumber
\end{equation}
Theorem 2.2 implies isomorphisms
$$(\pi_{\gamma_{1},\gamma_{2},N})_{!}\varphi_{(\gamma_{1},\gamma_{2})/u,N}(\pi_{\gamma_{1},\gamma_{2},N})^{*}[-1]\simeq(\pi_{\gamma_{1},\gamma_{2},N})_{!}(\pi_{\gamma_{1},\gamma_{2},N})^{*}(i_{\gamma_{1},\gamma_{2},N})_{*}(i_{\gamma_{1},\gamma_{2},N})^{*}$$
and
\begin{equation}
\begin{array}{ll}
&(q_{\overline{Q},\gamma_{1},\gamma_{2},N})_{!}(\pi_{\gamma_{1},\gamma_{2},\gamma_{1}\times\gamma_{2},N})_{!}\varphi_{(\gamma_{1},\gamma_{2},\gamma_{1}\times\gamma_{2})/u,N}(\pi_{\gamma_{1},\gamma_{2},\gamma_{1}\times\gamma_{2},N})^{*}(q_{\overline{Q},\gamma_{1},\gamma_{2},N})^{*}[-1]
\\\simeq&(q_{\overline{Q},\gamma_{1},\gamma_{2},N})_{!}(\pi_{\gamma_{1},\gamma_{2},\gamma_{1}\times\gamma_{2},N})_{!}(\pi_{\gamma_{1},\gamma_{2},\gamma_{1}\times\gamma_{2},N})^{*}(i_{\gamma_{1},\gamma_{2},\gamma_{1}\times\gamma_{2},N})_{*}(i_{\gamma_{1},\gamma_{2},\gamma_{1}\times\gamma_{2},N})^{*}(q_{\overline{Q},\gamma_{1},\gamma_{2},N})^{*}.$$
\end{array}\nonumber
\end{equation}
Pulling back to
$\textbf{M}_{\overline{Q},\gamma_{1},\gamma_{2},N}^{sp}\times\mathbb{C}^{*}$
gives us the commutative diagram
$$\begin{xy}
(0,20)*+{H_{c,\textbf{G}_{\gamma_{1},\gamma_{2}}}^{\bullet,crit}(\widetilde{\textbf{M}}_{\widehat{Q},\gamma_{1},\gamma_{2}}^{sp},W_{\gamma_{1},\gamma_{2}})}="v1";
(80,20)*+{H_{c,\textbf{G}_{\gamma_{1}}\times
\textbf{G}_{\gamma_{2}}}^{\bullet,crit}(\widetilde{\textbf{M}}_{\widehat{Q},\gamma_{1},\gamma_{2}}^{sp},W_{\gamma_{1},\gamma_{2}})\otimes\mathbb{T}^{-\gamma_{1}\cdot\gamma_{2}}}="v2";
(0,0)*+{H_{c,\textbf{G}_{\gamma_{1},\gamma_{2}}}^{\bullet}(\widetilde{\textbf{M}}_{\Pi_{Q},\gamma_{1},\gamma_{2}}^{sp},\mathbb{Q})\otimes\mathbb{T}^{\gamma\cdot\gamma-\gamma_{1}\cdot\gamma_{2}}}="v3";
(80,0)*+{H_{c,\textbf{G}_{\gamma_{1}}\times
\textbf{G}_{\gamma_{2}}}^{\bullet}(\widetilde{\textbf{M}}_{\Pi_{Q},\gamma_{1},\gamma_{2}}^{sp},\mathbb{Q})\otimes\mathbb{T}^{\gamma\cdot\gamma-2\gamma_{1}\cdot\gamma_{2}}}="v4";
{\ar@{->}^{\sim} "v1";"v2"}; {\ar@{->}^{\sim} "v3";"v4"};
{\ar@{->}^{\wr} "v1";"v3"}; {\ar@{->}^{\wr} "v2";"v4"}
\end{xy}$$

\item[$\bullet$]$H_{c,\textbf{G}_{\gamma_{1}}\times\textbf{G}_{\gamma_{2}}}^{\bullet}(\widetilde{\textbf{M}}_{\Pi_{Q},\gamma_{1},\gamma_{2}}^{sp},\mathbb{Q})\stackrel{\sim}{\longrightarrow}
H_{c,\textbf{G}_{\gamma_{1}}\times
\textbf{G}_{\gamma_{2}}}^{\bullet}(\textbf{M}_{\Pi_{Q},\gamma_{1}}^{sp}\times\textbf{M}_{\Pi_{Q},\gamma_{2}}^{sp},\mathbb{Q})\otimes\mathbb{T}^{-\chi_{\widetilde{Q}}(\gamma_{2},\gamma_{1})}$.
Similar as the previous step, the affine fibrations
$p_{\gamma_{1},\gamma_{2},N}:
\overline{(\textbf{M}_{\widehat{Q},\gamma_{1},\gamma_{2}},\textbf{G}_{\gamma_{1}}\times
\textbf{G}_{\gamma_{2}})}_{N}\rightarrow\overline{(\textbf{M}_{\widehat{Q},\gamma_{1}}\times\textbf{M}_{\widehat{Q},\gamma_{2}},\textbf{G}_{\gamma_{1}}\times
\textbf{G}_{\gamma_{2}})}_{N}$ induce isomorphisms
\\$(p_{\gamma_{1},\gamma_{2},N})_{!}\varphi_{\gamma_{1},\gamma_{2},\gamma_{1}\times\gamma_{2},N}(p_{\gamma_{1},\gamma_{2},N})^{*}\mathbb{Q}_{\overline{(\textbf{M}_{\widehat{Q},\gamma_{1}}\times\textbf{M}_{\widehat{Q},\gamma_{2}},\textbf{G}_{\gamma_{1}}\times
\textbf{G}_{\gamma_{2}})}_{N}\times\mathbb{C}^{*}}\\\stackrel{\sim}{\longrightarrow}\varphi_{\gamma_{1}\boxplus\gamma_{2},N}(\mathbb{Q}_{\overline{(\textbf{M}_{\widehat{Q},\gamma_{1}}\times\textbf{M}_{\widehat{Q},\gamma_{2}},\textbf{G}_{\gamma_{1}}\times
\textbf{G}_{\gamma_{2}})}_{N}\times\mathbb{C}^{*}}\otimes\mathbb{T}^{l})$,where
$l=\sum\limits_{a:i\rightarrow
j\in\widetilde{Q}_{1}}\gamma_{1}^{j}\gamma_{2}^{i}$. Then we have
isomorphisms
\begin{equation}
\begin{array}{ll}
&(\pi_{\gamma_{1}\times\gamma_{2},N})_{!}(p_{\gamma_{1},\gamma_{2},N})_{!}\varphi_{(\gamma_{1},\gamma_{2},\gamma_{1}\times\gamma_{2})/u,N}(p_{\gamma_{1},\gamma_{2},N})^{*}(\pi_{\gamma_{1}\times\gamma_{2},N})^{*}\mathbb{Q}_{\overline{(\textbf{M}_{\overline{Q},\gamma_{1}}\times\textbf{M}_{\overline{Q},\gamma_{2}},\textbf{G}_{\gamma_{1}}\times
\textbf{G}_{\gamma_{2}})}_{N}\times\mathbb{C}^{*}}
\\&\stackrel{\sim}{\longrightarrow}(\pi_{\gamma_{1}\times\gamma_{2},N})_{!}\varphi_{(\gamma_{1}\boxplus\gamma_{2})/u,N}(\pi_{\gamma_{1}\times\gamma_{2},N})^{*}(\mathbb{Q}_{\overline{(\textbf{M}_{\overline{Q},\gamma_{1}}\times\textbf{M}_{\overline{Q},\gamma_{2}},\textbf{G}_{\gamma_{1}}\times
\textbf{G}_{\gamma_{2}})}_{N}\times\mathbb{C}^{*}}\otimes\mathbb{T}^{l}).
\end{array}\nonumber
\end{equation}
The commutative diagram $$\begin{xy}
(0,20)*+{\overline{(\textbf{M}_{\widehat{Q},\gamma_{1},\gamma_{2}},\textbf{G}_{\gamma_{1}}\times\textbf{G}_{\gamma_{2}})}_{N}\times\mathbb{C}^{*}}="v1";
(80,20)*+{\overline{(\textbf{M}_{\widehat{Q},\gamma_{1}}\times\textbf{M}_{\widehat{Q},\gamma_{2}},\textbf{G}_{\gamma_{1}}\times\textbf{G}_{\gamma_{2}})}_{N}\times\mathbb{C}^{*}}="v2";
(0,0)*+{\overline{(\textbf{M}_{\overline{Q},\gamma_{1},\gamma_{2}},\textbf{G}_{\gamma_{1}}\times\textbf{G}_{\gamma_{2}})}_{N}\times\mathbb{C}^{*}}="v3";
(80,0)*+{\overline{(\textbf{M}_{\overline{Q},\gamma_{1}}\times\textbf{M}_{\overline{Q},\gamma_{2}},\textbf{G}_{\gamma_{1}}\times\textbf{G}_{\gamma_{2}})}_{N}\times\mathbb{C}^{*}}="v4";
{\ar@{->}^{p_{\gamma_{1},\gamma_{2},N}} "v1";"v2"};
{\ar@{->}^{p_{\overline{Q},\gamma_{1},\gamma_{2},N}}
"v3";"v4"};{\ar@{->}^{\pi_{\gamma_{1},\gamma_{2},\gamma_{1}\times\gamma_{2},N}}
"v1";"v3"}; {\ar@{->}^{\pi_{\gamma_{1}\times\gamma_{2},N}}
"v2";"v4"}
\end{xy}$$
implies isomorphisms \begin{equation}
\begin{array}{ll}
&(\pi_{\gamma_{1}\times\gamma_{2},N})_{!}(p_{\gamma_{1},\gamma_{2},N})_{!}\varphi_{(\gamma_{1},\gamma_{2},\gamma_{1}\times\gamma_{2})/u,N}(p_{\gamma_{1},\gamma_{2},N})^{*}(\pi_{\gamma_{1}\times\gamma_{2},N})^{*}[-1]
\\\simeq&(p_{\overline{Q},\gamma_{1},\gamma_{2},N})_{!}(\pi_{\gamma_{1},\gamma_{2},\gamma_{1}\times\gamma_{2},N})_{!}\varphi_{(\gamma_{1},\gamma_{2},\gamma_{1}\times\gamma_{2})/u,N}(\pi_{\gamma_{1},\gamma_{2},\gamma_{1}\times\gamma_{2},N})^{*}(p_{\overline{Q},\gamma_{1},\gamma_{2},N})^{*}[-1].
\end{array}\nonumber
\end{equation}
By  Theorem 2.2, we have
$$(\pi_{\gamma_{1}\times\gamma_{2},N})_{!}\varphi_{(\gamma_{1}\boxplus\gamma_{2})/u,N}(\pi_{\gamma_{1}\times\gamma_{2},N})^{*}[-1]\simeq(\pi_{\gamma_{1}\times\gamma_{2},N})_{!}(\pi_{\gamma_{1}\times\gamma_{2},N})^{*}(i_{\gamma_{1}\times\gamma_{2},N})_{*}(i_{\gamma_{1}\times\gamma_{2},N})^{*}$$
and
\begin{equation}
\begin{array}{ll}
&(p_{\overline{Q},\gamma_{1},\gamma_{2},N})_{!}(\pi_{\gamma_{1},\gamma_{2},\gamma_{1}\times\gamma_{2},N})_{!}\varphi_{(\gamma_{1},\gamma_{2},\gamma_{1}\times\gamma_{2})/u,N}(\pi_{\gamma_{1},\gamma_{2},\gamma_{1}\times\gamma_{2},N})^{*}(p_{\overline{Q},\gamma_{1},\gamma_{2},N})^{*}[-1]
\\\simeq&(p_{\overline{Q},\gamma_{1},\gamma_{2},N})_{!}(\pi_{\gamma_{1},\gamma_{2},\gamma_{1}\times\gamma_{2},N})_{!}(\pi_{\gamma_{1},\gamma_{2},\gamma_{1}\times\gamma_{2},N})^{*}(i_{\gamma_{1},\gamma_{2},\gamma_{1}\times\gamma_{2},N})_{*}(i_{\gamma_{1},\gamma_{2},\gamma_{1}\times\gamma_{2},N})^{*}(p_{\overline{Q},\gamma_{1},\gamma_{2},N})^{*}.$$
\end{array}\nonumber
\end{equation}
By pulling back to
$\overline{(\textbf{M}_{\overline{Q},\gamma_{1}}^{sp}\times\textbf{M}_{\overline{Q},\gamma_{2}}^{sp},\textbf{G}_{\gamma_{1}}\times\textbf{G}_{\gamma_{2}})}_{N}\times\mathbb{C}^{*}$,
we have
$$\begin{xy}
(0,20)*+{H_{c,\textbf{G}_{\gamma_{1}}\times\textbf{G}_{\gamma_{2}}}^{\bullet,crit}(\widetilde{\textbf{M}}_{\widehat{Q},\gamma_{1},\gamma_{2}}^{sp},W_{\gamma_{1},\gamma_{2}})}="v1";
(85,20)*+{H_{c,\textbf{G}_{\gamma_{1}}\times\textbf{G}_{\gamma_{2}}}^{\bullet,crit}(\textbf{M}_{\widehat{Q},\gamma_{1}}^{sp}\times\textbf{M}_{\widehat{Q},\gamma_{2}}^{sp},W_{\gamma_{1}}\boxplus
W_{\gamma_{2}})\otimes\mathbb{T}^{l}}="v2";
(0,0)*+{H_{c,\textbf{G}_{\gamma_{1}}\times\textbf{G}_{\gamma_{2}}}^{\bullet}(\widetilde{\textbf{M}}_{\Pi_{Q},\gamma_{1},\gamma_{2}}^{sp},\mathbb{Q})\otimes\mathbb{T}^{l_{1}}}="v3";
(85,0)*+{H_{c,\textbf{G}_{\gamma_{1}}\times\textbf{G}_{\gamma_{2}}}^{\bullet}(\textbf{M}_{\Pi_{Q},\gamma_{1}}^{sp}\times\textbf{M}_{\Pi_{Q},\gamma_{2}}^{sp},\mathbb{Q})\otimes\mathbb{T}^{l_{2}}}="v4";
{\ar@{->}^{\sim} "v1";"v2"}; {\ar@{->}^{\sim} "v3";"v4"};
{\ar@{->}^{\wr} "v1";"v3"}; {\ar@{->}^{\wr} "v2";"v4"}
\end{xy}$$

where $l_{1}=\gamma\cdot\gamma-\gamma_{1}\cdot\gamma_{2}$, and
$l_{2}=\gamma_{1}\cdot\gamma_{1}+\gamma_{2}\cdot\gamma_{2}+l$.

\item[$\bullet$]$H_{c,\textbf{G}_{\gamma_{1}}\times\textbf{G}_{\gamma_{2}}}^{\bullet}(\textbf{M}_{\Pi_{Q},\gamma_{1}}^{sp}\times\textbf{M}_{\Pi_{Q},\gamma_{2}}^{sp},\mathbb{Q})\stackrel{\sim}{\longrightarrow}H_{c,\textbf{G}_{\gamma_{1}}}^{\bullet}(\textbf{M}_{\Pi_{Q},\gamma_{1}}^{sp},\mathbb{Q})\otimes
H_{c,\textbf{G}_{\gamma_{2}}}^{\bullet}(\textbf{M}_{\Pi_{Q},\gamma_{2}}^{sp},\mathbb{Q})$.
This is the K\"unneth isomorphism compatible with the Thom-Sebastiani
isomorphism by Theorem 2.3.

\end{itemize}

The above computations can be summarized for convience of the reader in the form of the following statement.

\begin{prop}

The coproduct making the vector space
$\bigoplus\limits_{\gamma\in\mathbb{Z}^{I}_{\geqslant0}}H_{c,\textbf{G}}^{\bullet}(\textbf{M}_{\Pi_{Q},\gamma}^{sp},\mathbb{Q})$
into a coalgebra
is given by the composition of the maps

\begin{equation}
\begin{array}{ll}
&H_{c,\textbf{G}_{\gamma}}^{\bullet}(\textbf{M}_{\Pi_{Q},\gamma}^{sp},\mathbb{Q})\rightarrow
H_{c,\textbf{G}_{\gamma_{1},\gamma_{2}}}^{\bullet}(\textbf{M}_{\Pi_{Q},\gamma}^{sp},\mathbb{Q})\\
\rightarrow&H_{c,\textbf{G}_{\gamma_{1},\gamma_{2}}}^{\bullet}(\widetilde{\textbf{M}}_{\Pi_{Q},\gamma_{1},\gamma_{2}}^{sp},\mathbb{Q})\otimes\mathbb{T}^{-\gamma_{1}\cdot\gamma_{2}}
\\\stackrel{\sim}{\longrightarrow}&H_{c,\textbf{G}_{\gamma_{1}}\times\textbf{G}_{\gamma_{2}}}^{\bullet}(\widetilde{\textbf{M}}_{\Pi_{Q},\gamma_{1},\gamma_{2}}^{sp},\mathbb{Q})\otimes\mathbb{T}^{-2\gamma_{1}\cdot\gamma_{2}}
\\\stackrel{\sim}{\longrightarrow}&H_{c,\textbf{G}_{\gamma_{1}}\times\textbf{G}_{\gamma_{2}}}^{\bullet}(\textbf{M}_{\Pi_{Q},\gamma_{1}}^{sp}\times\textbf{M}_{\Pi_{Q},\gamma_{2}}^{sp},\mathbb{Q})\otimes\mathbb{T}^{-\chi_{Q}(\gamma_{1},\gamma_{2})-\chi_{Q}(\gamma_{2},\gamma_{1})}
\\\stackrel{\sim}{\longrightarrow}&H_{c,\textbf{G}_{\gamma_{1}}}^{\bullet}(\textbf{M}_{\Pi_{Q},\gamma_{1}}^{sp},\mathbb{Q})\otimes
H_{c,\textbf{G}_{\gamma_{2}}}^{\bullet}(\textbf{M}_{\Pi_{Q},\gamma_{2}}^{sp},\mathbb{Q})\otimes\mathbb{T}^{-\chi_{Q}(\gamma_{1},\gamma_{2})-\chi_{Q}(\gamma_{2},\gamma_{1})}.
\end{array}\nonumber
\end{equation}

\end{prop}

Now let
$\mathcal{H}_{\gamma}:=H_{c,\textbf{G}_{\gamma}}^{\bullet}(\textbf{M}_{\Pi_{Q},\gamma}^{sp},\mathbb{Q})^{\vee}\otimes\mathbb{T}^{-\chi_{Q}(\gamma,\gamma)}$,
and
$\mathcal{H}=\bigoplus\limits_{\gamma\in\mathbb{Z}^{I}_{\geqslant0}}\mathcal{H}_{\gamma}$.
Then the above coproduct makes $\mathcal{H}$ an associative algebra
with product
\begin{equation}
\begin{array}{ll}
&\mathcal{H}_{\gamma_{1}}\otimes\mathcal{H}_{\gamma_{2}}=H_{c,\textbf{G}_{\gamma_{1}}}^{\bullet}(\textbf{M}_{\Pi_{Q},\gamma_{1}}^{sp},\mathbb{Q})^{\vee}\otimes\mathbb{T}^{-\chi_{Q}(\gamma_{1},\gamma_{1})}\otimes
H_{c,\textbf{G}_{\gamma_{2}}}^{\bullet}(\textbf{M}_{\Pi_{Q},\gamma_{2}}^{sp},\mathbb{Q})^{\vee}\otimes\mathbb{T}^{-\chi_{Q}(\gamma_{2},\gamma_{2})}
\\=&H_{c,\textbf{G}_{\gamma_{1}}}^{\bullet}(\textbf{M}_{\Pi_{Q},\gamma_{1}}^{sp},\mathbb{Q})^{\vee}\otimes
H_{c,\textbf{G}_{\gamma_{2}}}^{\bullet}(\textbf{M}_{\Pi_{Q},\gamma_{2}}^{sp},\mathbb{Q})^{\vee}\otimes\mathbb{T}^{-\chi_{Q}(\gamma_{1},\gamma_{1})-\chi_{Q}(\gamma_{2},\gamma_{2})}
\\\rightarrow&H_{c,\textbf{G}_{\gamma_{1}+\gamma_{2}}}^{\bullet}(\textbf{M}_{\Pi_{Q},\gamma_{1}+\gamma_{2}}^{sp},\mathbb{Q})^{\vee}\otimes\mathbb{T}^{-\chi_{Q}(\gamma_{1},\gamma_{2})-\chi_{Q}(\gamma_{2},\gamma_{1})}\otimes\mathbb{T}^{-\chi_{Q}(\gamma_{1},\gamma_{1})-\chi_{Q}(\gamma_{2},\gamma_{2})}
\\=&H_{c,\textbf{G}_{\gamma_{1}+\gamma_{2}}}^{\bullet}(\textbf{M}_{\Pi_{Q},\gamma_{1}+\gamma_{2}}^{sp},\mathbb{Q})^{\vee}\otimes\mathbb{T}^{-\chi_{Q}(\gamma_{1}+\gamma_{2},\gamma_{1}+\gamma_{2})}=\mathcal{H}_{\gamma_{1}+\gamma_{2}}.
\end{array}\nonumber
\end{equation}

\begin{defn}
The associative algebra $\mathcal{H}$ is called the {\it
Cohomological Hall algebra}  of the
preprojective algebra $\Pi_{Q}$ associated with the quiver $Q$.
\end{defn}

\begin{rem}
In the framework of equivariant $K$-theory a similar notion was introduced in \cite{YaZha}.
\end{rem}

\begin{cor}
This product preserves the modified cohomological
degree, thus the zero degree part
$\mathcal{H}^0=\bigoplus\limits_{\gamma\in\mathbb{Z}^{I}_{\geqslant0}}\mathcal{H}_{\gamma}^0=\bigoplus\limits_{\gamma\in\mathbb{Z}^{I}_{\geqslant0}}H_{c,\textbf{G}_{\gamma}}^{-2\chi_{Q}(\gamma,\gamma)}(\textbf{M}_{\Pi_{Q},\gamma}^{sp},\mathbb{Q})^{\vee}\otimes\mathbb{T}^{-\chi_{Q}(\gamma,\gamma)}$
is a subalgebra of $\mathcal{H}$.
\end{cor}

We can reformulate the definition of COHA of $\Pi_{Q}$ using
language of stacks. The natural morphism of stacks
$\textbf{M}_{\Pi_{Q},\gamma_{1},\gamma_{2}}/\textbf{G}_{\gamma_{1},\gamma_{2}}\rightarrow\textbf{M}_{\Pi_{Q},\gamma}/\textbf{G}_{\gamma}$
is proper, hence it induces the pushforward map on $\mathcal{H}$.
Composting it with the pullback by the morphism
$\textbf{M}_{\Pi_{Q},\gamma_{1},\gamma_{2}}/\textbf{G}_{\gamma_{1},\gamma_{2}}\rightarrow\textbf{M}_{\Pi_{Q},\gamma_{1}}/\textbf{G}_{\gamma_{1}}\times\textbf{M}_{\Pi_{Q},\gamma_{2}}/\textbf{G}_{\gamma_{2}}$,
we obtain the product.

\subsection{Lusztig's seminilpotent Lagrangian subvariety}
In this subsection we work in the framework close to the one
from \cite{B}.

Let $Q$ be a quiver (possibly with loops) with vertices $I$ and
arrows $\Omega$, and denote by $\Omega_{i}$ the set of loops at
$i\in I$. We call $i$ imaginary if the number of loops
$\omega_{i}=|\Omega_{i}|\geqslant1$, and real if $\omega_{i}=0$. Let
$I^{im}$ be the set of imaginary vertices and $I^{re}$ real
vertices.

\begin{defn}
A representation $x\in\textbf{M}_{\overline{Q},\gamma}$ is {\it
seminilpotent} if there is an $I$-graded filtration
$W=(W_{0}=V_{\gamma}\supset\ldots\supset W_{r}=\{0\})$ of the
representation space $V_{\gamma}=(V_{i})_{i\in I}$, such that
$x_{a^{*}}(W_{\bullet})\subseteq W_{\bullet+1}$, and
$x_{a}(W_{\bullet})\subseteq W_{\bullet}$ for $a\in\Omega$.
\end{defn}

\begin{rem}
Our definition of seminilpotency is slightly different from that in {\rm
\cite{B}}. We put nilpotent condition on the dual arrows $a^{*}$
rather than $a$. But main results of {\rm \cite{B}}  hold
in our situation as well.
\end{rem}

We denote by $\textbf{M}_{\overline{Q},\gamma}^{sp}$ the space of
seminilpotent representations of dimension $\gamma$. Then by
\cite[Th.~1.15]{B}, the space of seminipotnet representations of
$\Pi_{Q}$ of dimension $\gamma$,
$\textbf{M}_{\Pi_{Q},\gamma}^{sp}\subset\textbf{M}_{\overline{Q},\gamma}^{sp}$,
is a Lagrangian subvariety of $\textbf{M}_{\overline{Q},\gamma}$.

Let
$\textbf{M}_{\Pi_{Q},\gamma,i,l}^{sp}=\{x\in\textbf{M}_{\Pi_{Q},\gamma}^{sp}|{\rm
codim}(\bigoplus\limits_{j\neq i,a:j\rightarrow i {\rm in}
\overline{Q}}{\rm Im} x_{a})=l\}$. Then
$\textbf{M}_{\Pi_{Q},\gamma}^{sp}=\bigcup\limits_{i\in I,
l\geqslant1}\textbf{M}_{\Pi_{Q},\gamma,i,l}^{sp}$ by the
seminilpotency condition. There is a one to one correspondence of
the sets of irreducible components (see \cite[Prop.1.14]{B})
\begin{equation}
{\rm
Irr}(\textbf{M}_{\Pi_{Q},\gamma,i,l}^{sp})\stackrel{\sim}{\longrightarrow}{\rm
Irr} (\textbf{M}_{\Pi_{Q},\gamma-le_{i},i,0}^{sp})\times{\rm
Irr}(\textbf{M}_{\Pi_{Q},le_{i}}^{sp}),
\end{equation}
where $e_{i}=(\delta_{ij})_{j\in I}$. For any vertex $i$, we have
${\rm
Irr}(\textbf{M}_{\Pi_{Q},\gamma}^{sp})=\bigsqcup\limits_{l\geqslant0}{\rm
Irr}(\textbf{M}_{\Pi_{Q},\gamma,i,l}^{sp})$. If $i\in I^{re}$ then
${\rm Irr}(\textbf{M}_{\Pi_{Q},le_{i}}^{sp})$ consists of only one
element, namely the zero representation. We denote by $Z_{i,l}$ the
only element in ${\rm Irr}(\textbf{M}_{\Pi_{Q},le_{i}}^{sp})$. If
$i\in I^{im}$, then there are two cases. If the number of loops
$\omega_{i}=1$, then ${\rm Irr}(\textbf{M}_{\Pi_{Q},le_{i}}^{sp})$
is parametrized by $\mathfrak{C}_{i,l}=\{c=(c_{k})\}$, the set of
partitions of $l$ (i.e., $\sum_{k}c_{k}=l$, $c_{k}>0,\forall k$, and
$c_{k+1}\geqslant c_{k}$). If $\omega_{i}>1$, then it is
parametrized by the set of compositions also denoted by
$\mathfrak{C}_{i,l}$ (i.e., $\sum_{k}c_{k}=l$, $c_{k}>0,\forall k$).

We put $|c|=\sum_{k}c_{k}$ for $c\in\mathfrak{C}_{i,l}$, and denote
by $Z_{i,c}\in{\rm Irr}(\textbf{M}_{\Pi_{Q},le_{i}}^{sp})$ the
irreducible component corresponding to $c$. Let  $Z\in {\rm
Irr}(\textbf{M}_{\Pi_{Q},\gamma}^{sp})$, then there exists $i\in I$
and $l\geqslant1$ such that
$Z\bigcap\textbf{M}_{\Pi_{Q},\gamma,i,l}^{sp}$ is dense in $Z$. We
denote by $\varepsilon_{i}(Z)$ the corresponding partition or
composition if $i\in I^{im}$, and $\varepsilon_{i}(Z)=l$ if $i\in
I^{re}$, via the one to one correspondence (4).

Now let $\mathscr{M}_{\gamma}$ be the $\mathbb{Q}$-vector space of
constructible functions
$f:\textbf{M}_{\Pi_{Q},\gamma}^{sp}\rightarrow \mathbb{Q}$ which are
constant on any $\textbf{G}_{\gamma}$-orbit, and
$\mathscr{M}=\bigoplus_{\gamma}\mathscr{M}_{\gamma}$. Then one can
define a product $\ast$ on $\mathscr{M}$ in the way which is
analogous to the definition of Lusztig for nilponent case in
\cite[Section~12]{L1}.

More precisely, let us denote by $\textbf{M}_{\Pi_{Q},V}^{sp}$ the
space of seminilpotent representations of $\Pi_{Q}$ with $I$-graded
vector space $V$, and $\mathscr{M}_{V}$ the $\mathbb{Q}$-vector
space of constructible functions
$f:\textbf{M}_{\Pi_{Q},V}^{sp}\rightarrow \mathbb{Q}$ constant on
any $\textbf{G}_{\gamma}$-orbit. Let $V_{1}$, $V_{2}$ and $V$ be
$I$-graded vector spaces of dimensions $\gamma_{1}$, $\gamma_{2}$
and $\gamma=\gamma_{1}+\gamma_{2}$ respectively, and
$f_{i}\in\mathscr{M}_{V_{i}},i=1,2$. Then $f_{1}\ast
f_{2}\in\mathscr{M}_{V}$ is defined using the diagram
$$\begin{xy}
(0,0)*+{\textbf{M}_{\Pi_{Q},V_{1}}^{sp}\times\textbf{M}_{\Pi_{Q},V_{2}}^{sp}}="v1";
(40,0)*+{\textbf{F}'}="v2"; (60,0)*+{\textbf{F}''}="v3";
(90,0)*+{\textbf{M}_{\Pi_{Q},V}^{sp}}="v4"; {\ar@{->}_{p_{1}}
"v2";"v1"}; {\ar@{->}^{p_{2}} "v2";"v3"}; {\ar@{->}^{p_{3}}
"v3";"v4"}
\end{xy}$$
where the notations are as follows: $F''$ is the variety of pairs
$(x,U)$ with $x\in\textbf{M}_{\Pi_{Q},V}^{sp}$ and $U$ an $x$-stable
$I$-graded subspace of $V$ with dimension $\gamma_{2}$; $F'$ is the
variety of quadruples $(x,U,R'',R')$ where $(x,U)\in F''$,
$R'':V_{2}\stackrel{\sim}{\longrightarrow}U$ and
$R':V_{1}\stackrel{\sim}{\longrightarrow}V/U$. The maps
$p_{1}(x,U,R'',R')=(x_{1},x_{2})$ where $xR'=R'x_{1}$ and
$xR''=R''x_{2}$, $p_{2}(x,U,R'',R')=(x,U)$, and $p_{3}(x,U)=x$. Note
that $p_{2}$ is a
$\textbf{G}_{V_{1}}\times\textbf{G}_{V_{2}}$-princepal bundle and
$p_{3}$ is proper. Let $f(x_{1},x_{2})=f_{1}(x_{1})f_{2}(x_{2})$,
then there is a unique function $f_{3}\in\mathscr{M}_{F''}$ such
that $p_{1}^{*}f=p_{2}^{*}f_{3}$. Finally, define $f_{1}\ast
f_{2}=(p_{3})_{!}(f_{3})$. By identifying the vector spaces
$\mathscr{M}_{V}$ for various $V$ with $\mathscr{M}_{\gamma}$ in a
coherent way ($\dim (V)=\gamma$), we define the product $\ast$ on
$\mathscr{M}$, making it an associative $\mathbb{Q}$-algebra.

One can also reformulate this product using the diagram
of stacks
$$\textbf{M}_{\Pi_{Q},\gamma_{2}}/\textbf{G}_{\gamma_{2}}\times\textbf{M}_{\Pi_{Q},\gamma_{1}}/\textbf{G}_{\gamma_{1}}\leftarrow\textbf{M}_{\Pi_{Q},\gamma_{2},\gamma_{1}}/\textbf{G}_{\gamma_{2},\gamma_{1}}\rightarrow\textbf{M}_{\Pi_{Q},\gamma}/\textbf{G}_{\gamma}.$$

We denote by $1_{i,c}$ (resp. $1_{i,l}$) the characteristic function
of $Z_{i,c}$ (resp. $Z_{i,l}$), and
$\mathscr{M}_{0}\subseteq\mathscr{M}$ the subalgebra generated by
$1_{i,(l)}$ and $1_{i,1}$. For any $Z\in{\rm
Irr}(\textbf{M}_{\Pi_{Q},\gamma}^{sp})$ and
$f\in\mathscr{M}_{\gamma}$, let $\rho_{Z}(f)=c$ if $Z\bigcap
f^{-1}(c)$ is open dense in $Z$.

\begin{thm}(see {\rm \cite[Prop.~1.18]{B}})
For any $Z\in{\rm Irr}(\textbf{M}_{\Pi_{Q},\gamma}^{sp})$ there
exists
$f_{Z}\in\mathscr{M}_{0,\gamma}=\mathscr{M}_{0}\cap\mathscr{M}_{\gamma}$
such that $\rho_{Z}(f_{Z})=1$, and $\rho_{Z'}(f_{Z})=0$ for $Z'\neq
Z$.
\end{thm}

\subsection{Generalized quantum group}

We recall some definitions and facts about generalized quantum group
introduced in \cite{B}.

Let $(\bullet,\bullet)$ be the symmetric Euler form on
$\mathbb{Z}^{I}$ defined by $(i,j)=2\delta_{ij}-a_{ij}-a_{ji}$, and
$(\iota,j)=l(i,j)$ if $\iota=(i,l)\in
I_{\infty}=(I^{re}\times\{1\})\bigcup
(I^{im}\times\mathbb{N}_{\geqslant1})$ and $j\in I$.

\begin{defn}
Let $F$ be the $\mathbb{Q}(v)$-algebra generated by
$(E_{\iota})_{\iota\in I_{\infty}}$, $\mathbb{N}^{I}$-graded by
$|E_{\iota}|=li$ for $\iota=(i,l)$. If $A\subseteq\mathbb{N}^{I}$,
then let $\mathrm{F}[A]=\{E\in\mathrm{F}||E|\in A\}$.
\end{defn}

For any $\gamma=(\gamma^{i})_{i\in I}\in\mathbb{Z}^{I}$, let ${\rm
ht}(\gamma)=\sum_{i}\gamma^{i}$ be its height, and
$v_{\gamma}=\prod_{i} v_{i}^{\gamma^{i}}$, where
$v_{i}=v^{(i,i)/2}$. We endow $F$ with  a coproduct
$\delta(E_{i,l})=\sum\limits_{l_{1}+l_{2}=l}v_{i}^{l_{1}l_{2}}E_{i,l_{1}}E_{i,l_{2}}$,
where $E_{i,0}=1$. Then for any family $(v_{\iota})_{\iota\in
I_{\infty}}\subseteq\mathbb{Q}(v)$, there is a bilinear form
$\{\bullet,\bullet\}$ on $F$ such that
\begin{itemize}
\item[$\bullet$]$\{E,E'\}=0$ if $|E|\neq|E'|$,
\item[$\bullet$]$\{E_{\iota},E_{\iota}\}=v_{\iota}$, $\forall\iota\in
I_{\infty}$,
\item[$\bullet$]$\{EE',E''\}=\{E\otimes E',\delta(E'')\}$, $\forall E,E',E''\in {\rm
F}$.
\end{itemize}
It turns out that
$\sum\limits_{l_{1}+l_{2}=-(\iota,j)+1}(-1)^{l_{1}}\frac{E_{j,1}^{l_{1}}}{l_{1}!}E_{\iota}\frac{E_{j,1}^{l_{2}}}{l_{2}!}$
is in the radical of $\{\bullet,\bullet\}$.

\begin{defn}
Let $\widetilde{\mathcal {U}}^{+}$ be the quotient of $F$ by the
ideal generated by the above element and the commutators
$[E_{i,l},E_{i,k}]$ for $\omega_{i}=1$. Then $\{\bullet,\bullet\}$
is well-defined on $\widetilde{\mathcal{U}}^{+}$. Let
$\mathcal{U}^{+}$ be the quotient of $\widetilde{\mathcal{U}}^{+}$
by the radical of $\{\bullet,\bullet\}$.
\end{defn}

\begin{thm}(see {\rm \cite[Th.~3.34]{B}})
There is an isomorphism of algebras
\begin{equation}
\begin{array}{ll} \phi:&\mathcal{U}^{+}_{v=1}\rightarrow\mathscr{M}_{0},
\\&E_{i,(l)}\mapsto1_{i,(l)}, \quad i\in I^{im},\\&E_{i,1}\mapsto1_{i,1}, \quad i\in I^{re}.
\end{array}\nonumber
\end{equation}
\end{thm}

\begin{defn}
The  semicanonical basis of  $\,$
${\mathcal U}^{+}_{v=1}$ is
$\phi^{-1}(\{f_{Z}|Z\in{\rm Irr}(\textbf{M}_{\Pi_{Q}}^{sp})\})$.
\end{defn}

\subsection{Semicanonical basis of $\mathcal{H}^0$}

We have already seen that for an appropriate subspace
$\textbf{M}_{\overline{Q},\gamma}^{sp}\subset\textbf{M}_{\overline{Q},\gamma}$,
the degree $0$ part $\mathcal{H}^0$ of COHA is a subalgebra. In
particular, we can take $\textbf{M}_{\overline{Q},\gamma}^{sp}$ to
be the space of seminilpotent representations of $\overline{Q}$.
Then $\textbf{M}_{\Pi_{Q},\gamma}^{sp}$ is the space of
seminilpotent representations in $\textbf{M}_{\Pi_{Q},\gamma}$, and
${\rm
dim}(\textbf{M}_{\Pi_{Q},\gamma}^{sp}/\textbf{G}_{\gamma})=-\chi_{Q}(\gamma,\gamma)$,
so the classes of irreducible components $\{[Z]|Z\in{\rm
Irr}(\textbf{M}_{\Pi_{Q},\gamma}^{sp})\}$ lie in $\mathcal{H}^0$. In
fact, these classes form a basis of $\mathcal{H}^0$ by the following
theorem.

\begin{thm}
Let $X$ be a scheme with top dimensional irreducible components
$\{C^{k}\}$, and a connected algebraic group G acts on it. Then
$H_{c,G}^{2top}(X)$ has a basis one to one corresponding to
$\{C^{k}\}$, where top is the dimension of the stack $X/G$.
\end{thm}
\begin{proof} Choose an embedding $G\hookrightarrow GL(n,\mathbb{C})$. Let
$fr(n,N)$ be the space of $n$-tuples of linearly independent vectors
in $\mathbb{C}^{N}$ for $N\geqslant n$. Then $X\times fr(n,N)$ has
irreducible components $\{C^{k}\times fr(n,N)\}$, thus
$X\times_{G}fr(n,N)=(X\times fr(n,N))/G$ has irreducible components
$\{\overline{C^{k}}\}$ one to one corresponding to $\{C^{k}\}$ since
$G$ is irreducible. Then the Borel-Moore homology
$H_{2\bullet}^{BM}(X\times_{G}fr(n,N))$ has a basis
$\{[\overline{C^{k}}]\}$, where $\bullet=\dim (X)+\dim
(fr(n,N))-\dim G$, implying that
$H_{c}^{2\bullet}(X\times_{G}fr(n,N))^{\vee}=H_{2\bullet}^{BM}(X\times_{G}
fr(n,N))$ has basis one to one corresponding to $\{C^{k}\}$ (For
details of Borel-Moore homology, see \cite[Section~2.6]{CG}). Then
$H_{c,G}^{2top}(X)=\lim\limits_{N\rightarrow\infty}H_{c}^{2\bullet}(X\times_{G}fr(n,N))\otimes\mathbb{T}^{-\dim
fr(n,N)}$ has basis one to one corresponding to $\{C^{k}\}$, where
$top=\bullet-\dim (fr(n,N))=\dim (X/G)$.
\end{proof}

\begin{defn} We call the basis defined above the {\it
semicanonical basis} of the subalgebra $\mathcal{H}^0$.
\end{defn}

Given an element $\mathcal {F}$ in $\mathscr{D}^{b}(X)$ with
constructible cohomology, and $x\in X$, the function
$\chi(\mathcal{F})(x)=\chi(\mathcal{F}_{x})=\sum_{i}(-1)^{i}{\rm
dim}(H^{i}(\mathcal{F}_{x}))$ is constructible. Moreover, the standard
operations (pullback, pushforward, etc.) in $\mathscr{D}^{b}(X)$ and the corresponding operations on
constructible functions are compatible.

Recall the family of constructible functions $\{f_{Z}|Z\in {\rm
Irr}(\textbf{M}_{\Pi_{Q}}^{sp})\}$. Then $U_{Z}=f_{Z}^{-1}(1)$ is constructible. Let $f_{Z,N}$ be the characteristic function
of $\overline{(U_{Z},\textbf{G}_{\gamma})}_{N}$, and
$\mathbb{Q}_{Z,N}$ be the constant sheaf on
$\overline{(U_{Z},\textbf{G}_{\gamma})}_{N}$. Since the operations
on constructible functions and constructible sheaves agree, there is
an isomorphism of algebras
$\Psi:\mathcal{H}^0\rightarrow\mathscr{M}_{0}^{op},[Z]\mapsto
f_{Z}$. It is obtained by taking the dual of compactly supported
cohomology and passing to the limit.

Furthermore, notice that
$\mathcal{H}^0\simeq (\mathcal {U}^{+}_{v=1})^{op}$, and that
Lusztig's product $\ast$ is opposite to the product of COHA (see the
end of Section 2.3).

The semicanonical basis of $\mathcal{H}^0$ is compatible with a certain filtration.
More precisely, we have the following result.

\begin{thm}
Fix $d=(d_{i})\in\mathbb{Z}_{\geqslant0}^{I}$. Then the subspace
spanned by $\{[Z]|\exists i, s.t.|\varepsilon_{i}(Z)|\geqslant
d_{i}\}$ coincides with $\sum\limits_{i\in I,|c|=d_{i}}\mathcal
{H}^{0}[Z_{i,c}]$, where $Z_{i,c}\in{\rm
Irr}(\textbf{M}_{\Pi_{Q},le_{i}}^{sp})$ is the irreducible component
corresponding to $c$ (defined in Section 2.3), and $c=l$ if $i\in
I^{re}$.
\end{thm}

\begin{proof}
By definitions, $\sum\limits_{i\in I,|c|=d_{i}}\mathcal
{H}^{0}[Z_{i,c}]$ is contained in the subspace spanned by
$\{[Z]|\exists i, s.t.|\varepsilon_{i}(Z)|\geqslant d_{i}\}$. To
prove the reverse inclusion it suffices to show that for any $i\in
I$, $\gamma\in\mathbb{Z}_{\geqslant0}^{I}$, and
$[Z]\in\mathcal{H}^0$ such that $Z\in{\rm
Irr}(\textbf{M}_{\Pi_{Q},\gamma}^{sp})$ and
$|\varepsilon_{i}(Z)|=l$, we have $[Z]\in\sum\limits_{|c|=l}\mathcal
{H}^{0}[Z_{i,c}]$. We use descending induction on
$l\leqslant\gamma^{i}$. For above $Z$, we have
$\gamma-le_{i}\in\mathbb{N}^{I}$, and by the proof of
\cite[Pro.~1.18]{B}, there exists a unique $Z'\in{\rm
Irr}(\textbf{M}_{\Pi_{Q},\gamma-le_{i}}^{sp})$ and $Z_{i,c}\in{\rm
Irr}(\textbf{M}_{\Pi_{Q},le_{i}}^{sp})$ such that
$|\varepsilon_{i}(Z')|=0$ and
$[Z'][Z_{i,c}]=Z+\sum\limits_{|\varepsilon_{i}(\widetilde{Z})|>l}a_{\widetilde{Z}}[\widetilde{Z}]$
for some $a_{\widetilde{Z}}\in\mathbb{Q}$. By applying the induction
hypothesis to $\widetilde{Z}$ we have that the subspace spanned by
$\{[Z]|\exists i, s.t.|\varepsilon_{i}(Z)|\geqslant d_{i}\}$ is
contained in $\sum\limits_{i\in I,|c|=d_{i}}\mathcal
{H}^{0}[Z_{i,c}]$. Thus the two subspaces coincide.
\end{proof}

The dual of representations of $\Pi_{Q}$ induces a bijection $\ast:
{\rm Irr}(\textbf{M}_{\Pi_{Q},\gamma}^{sp})\rightarrow{\rm
Irr}(\textbf{M}_{\Pi_{Q},\gamma}^{sp}), Z\mapsto Z^{\ast}$, thus an
antiautomorphism of $\mathcal {H}^{0}$. Then the dual of the above
theorem holds:

\begin{thm}
The subspace spanned by $\{[Z]|\exists i,
s.t.|\varepsilon_{i}(Z^{\ast})|\geqslant d_{i}\}$ coincides with
$\sum\limits_{i\in I,|c|=d_{i}}[Z_{i,c}]\mathcal {H}^{0}$.
\end{thm}

\section{$2CY$ categories and Donaldson-Thomas series}

In Section 2 we discussed the semicanonical basis obtained as a result of the
dimensional reduction from $3CY$ category to the $2$-dimensional category.
In this section we are going to discuss DT-series for $2CY$ categories.

\subsection{Motivic stack functions and motivic Hall algebras: reminder}

Let $X$ be a constructible set over a field $\textbf{k}$ of
characteristic zero, $G$ an affine algebraic group acting on $X$.
In this section we are going to recall the definition of the abelian group of stack functions $Mot_{st}((X,G))$ following
\cite[Section~4]{KoSo2} (see also \cite{Jo1} for a different exposition).

Let us consider  the following 2-category of constructible
stacks over ${\bf k}$. Objects are pairs $(X,G)$, where
$X$ is a constructible set, and $G$ is an affine algebraic group
acting on it.  The category of
1-morphisms Hom$((X_{1},G_{1}),(X_{2},G_{2}))$ consists of pairs $(Z,f)$,
where $Z$ is a $G_{1}\times G_{2}$-constructible set such that
$\{e\}\times G_{2}$ acts freely on $Z$ in such a way that we have
the induced $G_{1}$-equivariant isomorphism $Z/G_{2}\simeq X_{1}$,
and $f:Z\rightarrow X_{2}$ is a $G_{1}\times G_{2}$-equivariant map
with trivial action of $G_{1}$ on $X_{2}$.
Furthermore, objects of Hom$((X_{1},G_{1}),(X_{2},G_{2}))$ form
naturally a groupoid. The 2-category of constructible
stacks carries a direct sum operation induced by disjoint union of
stacks.

After the above preliminaries we define the  group of  motivic stack functions $Mot_{st}((X,G))$ as the abelian group
generated by isomorphism classes of 1-morphisms of stacks
$[(Y,H)\rightarrow(X,G)]$ with the fixed target $(X,G)$, subject to
the relations
\begin{itemize}
\item[$\bullet$]$[((Y_{1},H_{1})\sqcup(Y_{2},H_{2}))\rightarrow(X,G)]=[(Y_{1},H_{1})\rightarrow(X,G)]+[(Y_{2},H_{2})\rightarrow(X,G)]$,
\item[$\bullet$]$[(Y_{2},H)\rightarrow(X,G)]=[(Y_{1}\times \mathbb{A}_{k}^{d},H)\rightarrow(X,G)]$ if $Y_{2}\rightarrow Y_{1}$
 is an $H$-equivariant constructible vector bundle of rank $d$.
\end{itemize}

One can define pullbacks, pushforwards and fiber products of elements of
$Mot_{st}((X,G))$ in the natural way (see loc.cit.).

%Let $(Z,f)\in{\rm
%Hom}((X_{1},G_{1}),(X_{2},G_{2}))$. Then we define
%\begin{itemize}
%\item[$\bullet$]$f_{!}: Mot_{st}((X_{1},G_{1}))\rightarrow Mot_{st}((X_{2},G_{2})),
%[(Z_{1},f_{1}):(Y,H)\rightarrow(X_{1},G_{1})]\mapsto[(Z,f)\circ(Z_{1},f_{1}):(Y,H)\rightarrow(X_{2},G_{2})]$,
%\item[$\bullet$]$f^{\ast}: Mot_{st}((X_{2},G_{2}))\rightarrow Mot_{st}((X_{1},G_{1})),
%[(Y,H)\rightarrow(X_{2},G_{2})]\mapsto[(Y,H)\times_{(X_{2},G_{2})}(X_{1},G_{1})\rightarrow(X_{1},G_{1})]$,
%\item[$\bullet$]$\cdot: Mot_{st}((X,G))\times Mot_{st}((X,G))\rightarrow Mot_{st}((X,G)),
%[(Y_{1},H_{1})\rightarrow(X,G)]\cdot[(Y_{2},H_{2})\rightarrow(X,G)]\mapsto[(Y_{1},H_{1})\times_{(X,G)}(Y_{2},H_{2})\rightarrow(X,G)]$.
%\end{itemize}

Let $\mathcal {C}$ be an ind-constructible locally regular (e.g. locally Artin) triangulated
$A_{\infty}$-category over a field ${\bf k}$ (see \cite{KoSo2}).
Then the stack of objects admits a
countable decomposition into the union of quotient stacks $\mathcal
{O}b(\mathcal{C})=\sqcup_{i\in I}(Y_{i}, GL(N_i))$, where $Y_i$ is a reduced
algebraic scheme acted by the group $GL(N_i)$.

\begin{defn} (cf. {\rm \cite{KoSo2}})
The motivic Hall algebra $H(\mathcal{C})$ is the
$Mot(Spec(\textbf{k}))-$module \\$\bigoplus_{i\in I}Mot_{st}(Y_{i},
GL(N_{i}))[\mathbb{L}^{n}, n<0]$ (i.e. we extend the direct sum of
the groups of motivic stack functions by adding negative powers of
the Lefschetz motive ${\mathbb L}$), endowed with the product
defined below.
\end{defn}

The product is defined as follows. Let us denote $\dim Ext^{i}(E,F)$
by $(E,F)_{i}$, and use the truncated Euler characteristic
$(E,F)_{\leq N}=\sum_{i\leq N}(-1)^{i}(E,F)_{i}$. Let $[\pi_{i}:
Y_{i}\rightarrow\mathcal {O}b(\mathcal{C})], i=1, 2$ be two elements
of $H(\mathcal{C})$, then for any $n\in\mathbb{Z}$ we have
constructible sets $$W_{n}=\{(y_{1}, y_{2}, \alpha)|y_{i}\in Y_{i},
\alpha\in Ext^{1}(\pi_{2}(y_{2}), \pi_{1}(y_{1})), (\pi_{2}(y_{2}),
\pi_{1}(y_{1}))_{\leq 0}=n\}.$$ Then
$[tot((\pi_{1}\times\pi_{2})^{\ast}(\mathcal {E}\mathcal {X}\mathcal
{T}^{1}))\rightarrow\mathcal
{O}b(\mathcal{C})]=\sum_{n\in\mathbb{Z}}[W_{n}\rightarrow\mathcal
{O}b(\mathcal{C})]$. Define the product $[Y_{1}\rightarrow\mathcal
{O}b(\mathcal{C})]\cdot[Y_{2}\rightarrow\mathcal
{O}b(\mathcal{C})]=\sum_{n\in\mathbb{Z}}[W_{n}\rightarrow\mathcal
{O}b(\mathcal{C})]\mathbb{L}^{-n}$, where the map
$W_{n}\rightarrow\mathcal {O}b(\mathcal{C})$ is given by $(y_{1},
y_{2}, \alpha)\mapsto Cone(\alpha:
\pi_{2}(y_{2})[-1]\rightarrow\pi_{1}(y_{1}))$.

\begin{thm} (see {\rm \cite[Prop.~10]{KoSo2}})
The algebra $H(\mathcal{C})$ is associative.
\end{thm}

For a constructible stability condition on $\mathcal{C}$ with an
ind-constructible class map $cl:
K_{0}(\mathcal{C})\rightarrow\Gamma$, a central charge $Z:
\Gamma\rightarrow\mathbb{C}$, a strict sector
$V\subset\mathbb{R}^{2}$ and a branch Log of the logarithm function
on $V$, we have (see \cite{KoSo2}) the category
$\mathcal{C}_{V}:=\mathcal{C}_{V, Log}$ generated by semistables with the central charge in $V$.  Then we define the corresponding
completed motivic Hall algebra
$\widehat{H}(\mathcal{C}_{V}):=\prod\limits_{\gamma\in(\Gamma\cap
C(V,Z,Q))\cup\{0\}}H(\mathcal{C}_{V}\cap cl^{-1}(\gamma))$. It
contains an invertible element $A_{V}^{Hall}=1+\cdots=\sum_{i\in
I}\textbf{1}_{(\mathcal {O}b(\mathcal{C}_{V})\cap Y_{i},
GL(N_{i}))}$, where $1$ comes from the zero object. The element
$A_V$ corresponds (roughly) to the sum over all isomorphism classes
of objects of ${\mathcal C}_V$, each counted with the weight given
by the inverse to the motive of the group of automorphisms.

\begin{thm} (see {\rm \cite[Prop.~11]{KoSo2}})
The elements $A_{V}^{Hall}$ satisfy the Factorization Property:
$$A_{V}^{Hall}=A_{V_{1}}^{Hall}\cdot A_{V_{2}}^{Hall}$$
for a strict sector $V=V_{1}\sqcup V_{2}$ (decomposition in the
clockwise order).
\end{thm}

Let's fix the following data:
\begin{itemize}
\item[(1)]a triple $(\Gamma, \langle\bullet,\bullet\rangle, Q)$ consisting of a
free abelian group $\Gamma$ of finite rank endowed with a bilinear
form $\langle\bullet,\bullet\rangle:
\Gamma\otimes\Gamma\rightarrow\mathbb{Z}$, and a quadratic form $Q$
on $\Gamma_{\mathbb{R}}=\Gamma\otimes\mathbb{R}$,
\item[(2)]an ind-constructible ,
$Gal(\overline{\textbf{k}}/\textbf{k})-$ equivariant homomorphism
$cl_{\overline{\textbf{k}}}:
K_{0}(\mathcal{C}(\overline{\textbf{k}}))\rightarrow\Gamma$
compatible with the Euler form of $\CC$ and the bilinear form
$\langle\bullet,\bullet\rangle$,
\item[(3)]a constructible stability condition $\sigma\in Stab(\mathcal{C},
cl)$ compatible with the quadratic form $Q$ in the sense that
$Q|_{Ker(Z)}<0$ and $Q(cl_{\overline{\textbf{k}}}(E))\geq0$,
$\forall E\in \mathcal{C}^{ss}(\overline{\textbf{k}})$.
\end{itemize}

We define the  quantum torus $\mathcal {R}_{\Gamma, R}$ over a given
commutative unital ring $R$ containing an invertible symbol
$\mathbb{L}^{\frac{1}{2}}$ as an $R$-linear associative algebra
$\mathcal {R}_{\Gamma,
R}:=\bigoplus_{\gamma\in\Gamma}R\cdot\widehat{e}_{\gamma}$, where
the generators $\widehat{e}_{\gamma}, \gamma\in\Gamma$ satisfy the
relations $\widehat{e}_{\gamma_{1}}\widehat{e}_{\gamma_{2}}={\mathbb
L}^{{1\over{2}}\langle \gamma_1,\gamma_2\rangle}
\widehat{e}_{\gamma_{1}+\gamma_{2}}$, $\widehat{e}_{0}=1$. For any
strict sector $V\subset\mathbb{R}^{2}$, we define the quantum torus
associated with $V$ by $\mathcal
{R}_{V,R}:=\prod\limits_{\gamma\in\Gamma\cap
C_{0}(V,Z,Q)}R\cdot\widehat{e}_{\gamma}$, where
$C_{0}(V,Z,Q):=C(V,Z,Q)\cup\{0\}$, and $C(V,Z,Q)$ is the convex cone
generated by
$S(V,Z,Q)=\{x\in\Gamma_{\mathbb{R}}\setminus\{0\}|Z(x)\in V,
Q(x)\geq0\}$.

In the case when ${\mathcal C}$ is a $3CY$ category, one can define a homomorphism from the algebra $\widehat{H}(\mathcal{C}_{V})$ to an appropriate motivic quantum torus (the word  ``motivic" here means that the coefficient ring $R$ is a certain ring of motivic functions).
This homomorphism was defined in \cite{KoSo2} via the motivic Milnor fiber of the potential
of the $3CY$ category. The notion of motivic DT-series was also introduced in the loc.cit.

It was later shown in \cite{KoSo3} that in the case of quivers with
potential one can define  motivic DT-series differently, using
equivariant critical cohomology (cf. our Section 2). In that case
instead of the motivic Hall algebra one uses COHA.

\subsection{A class of $2CY$ categories}

Let us consider a class of $2$-dimensional Calabi-Yau categories
$\mathcal{C}$  which are:
\begin{itemize}
\item[1)]Ind-constructible and locally ind-Artin in the sense of \cite{KoSo2}.
\item[2)]Endowed with a constructible homomorphism of abelian groups (class map) $cl:
K_{0}(\mathcal{C})\rightarrow\Gamma$, where
$\Gamma\simeq\mathbb{Z}^{I}$ carries a symmetric integer bilinear
form $\langle\bullet,\bullet\rangle$, and the class map $cl$
satisfies $\langle
cl(E),cl(F)\rangle=\chi(E,F):=\sum_{i\in\mathbb{Z}}(-1)^{i}\dim
Ext^{i}(E,F)$.
\item[3)]Generated by a spherical collection $\mathcal {E}=(E_{i})_{i\in
I}$ in the sense of loc. cit. such that
$cl(E_{i})\in\Gamma_{+}\simeq\mathbb{Z}_{\geq0}^{I}$. This means
that $Ext^{\bullet}(E_i,E_i)\simeq H^{\bullet}(S^2)$, and that $Ext^m(E_i,E_j)$ can be non-trivial for $m=1$ only as long as $i\ne j$.
\item[4)]For any $\gamma\in\Gamma_{+}$ the stack $\mathcal{C}_{\gamma}(\mathcal
{E})$ of objects $F$ of the heart of the $t$- structure
corresponding to $(E_{i})_{i\in I}$ such that $cl(F)=\gamma$ is a
countable disjoint union of Artin stacks of dimensions less  or
equal than $-\frac{1}{2}\langle\gamma,\gamma\rangle$.
\item[5)]For any strict sector $V\subset\mathbb{R}^{2}$ with the vertex
at $(0, 0)$, and a constructible stability central charge $Z:
\Gamma\rightarrow\mathbb{C}$ such that $\Im(Z(E_{i})):=Z(cl(E_i))\in
V, i\in I$, the stack of objects of the category $\mathcal{C}_{V}$
generated by semistable objects with central charges in $V$ is a
finite union of Artin stacks satisfying the inequality of 4) above.
\end{itemize}

With the category from our class one can associate a symmetric quiver
(vertices correspond to spherical objects $E_i$ and arrows correspond
to a basis in $Ext^1(E_i, E_j)$). Similarly to \cite{KoSo2}, Section 8 one
can prove a classification theorem for our categories in terms of Ginzburg
algebras associated with quivers. Many $2CY$ categories which appear in ``nature"  belong to
our class. For example, if $Q$ is not an ADE quiver, then the derived category
of finite-dimensional representations of $\Pi_Q$ belongs to our class.
Without any restrictions on $Q$ one can construct a $2CY$ category as the category of dg-modules over the corresponding Ginzburg algebra.

\subsection{Stability conditions and braid group action}

Assume that $\CC$ is a $2CY$ category from our class. We consider an open subset of the
space $Stab(\CC)$ of stability conditions which is defined as
$U:=\prod_{i\in I}(Im\,z_i>0)$, i.e. it is a product of upper-half
planes. A point $Z=(z_i)_{i\in I}\in U$ defines the central charge
$Z: \Gamma:=\mathbb{Z}^I\to \mathbb{C}$ which maps classes of
spherical generators to the open upper-half plane (hence the
stability condition is determined by $Z$ and the $t$-structure in
$\CC$ generated by $(E_i)_{i\in I}$).

Recall that with every $i_0\in I$ we can associate an
autoequivalence of $\CC$ (called {\it reflection functor}) by the
formula
$$R_{E_{i_0}}:F\mapsto Cone(Ext^{\bullet}(E_{i_0},F)\otimes F\to F).$$

Then $R_{E_{i_0}}(E_{i_0})=E_{i_0}[-1]$, and $R_{E_{i_0}}(E_{j}),
j\ne i_0$ is determined as the middle term in the extension
$$0\to E_j\to R_{E_{i_0}}(E_{j})\to E_{i_0}\otimes Ext^1(E_{i_0},E_j)\to 0.$$

The inverse reflection functor $R_{E_{i_0}}^{-1}$ is given by

$$R_{E_{i_0}}^{-1}(E_{i_0})=E_{i_{i_0}}[1],$$

$$0\to E_{i_0}\otimes Ext^1(E_{i_0},E_j)\to R_{E_{i_0}}^{-1}(E_{j})\to E_j\to 0.$$

Reflection functors $R_{E_i}, i\in I$ generate a subgroup
$Braid_{\CC}\subset Aut(\CC)$, which induces an action on
$Stab(\CC)$. The orbit $D:=Braid_{\CC}(U)\subset Stab(\CC)$ is the
union of consecutive ``chambers''  obtained one from another one by
reflection functor $R_{E_j}$. Such consecutive chambers have a
common real codimension one boundary singled out by the  condition
$Im\,Z(E_j)=0$.

\begin{rem}
The group $Braid_{\CC}$ plays a role of the braid group (or Weyl
group) in the theory of Kac-Moody algebras. If we add also the group
$\mathbb{Z}$ of shifts $F\mapsto F[n], n\in \mathbb{Z}$ then we
obtain an affine version of the braid group $Braid_{\CC}\times
\mathbb{Z}$. In some examples $\mathbb{Z}\subset Braid_{\CC}$.
\end{rem}

\subsection{Motivic DT-series for $2CY$ categories}

Let $\CC$  be an ind-constructible locally regular $2CY$ category over ${\bf k}$.
 Let us fix $R=Mot(Spec(\textbf{k}))[\mathbb{L}^{\frac{1}{2}},\mathbb{L}^{-1},[GL(n)]^{-1}_{n\geqslant 1}]$ as the ground ring
for the quantum torus $\mathcal {R}_{\Gamma, R}$. We will denote the
latter by $\mathcal {R}_{\Gamma}$. It is a commutative algebra
generated by elements $\widehat{e}_\gamma, \gamma\in \Gamma$ such
that
$\widehat{e}_{\gamma_1+\gamma_2}=\widehat{e}_{\gamma_1}\widehat{e}_{\gamma_2},
\widehat{e}_0=1$. Let us also fix a stability condition on $\CC$
with the central charge $Z:\Gamma\to \mathbb{C}$.

\begin{defn}
The motivic weight $\omega\in Mot(\mathcal {O}b(\mathcal{C}))$ is
defined by
$\omega(E)=\mathbb{L}^{\frac{1}{2}(\chi(E,E))}$.
\end{defn}

Then we have the following result.

\begin{prop}
The map $\Phi: H(\mathcal{C})\rightarrow\mathcal {R}_{\Gamma}$ given
by $\Phi(\nu)=(\nu,\omega)\widehat{e}_{\gamma}, \nu\in
H(\mathcal{C})_{\gamma}$ satisfies the condition
$\Phi(\nu_{1}\cdot\nu_{2})=\Phi(\nu_{1})\Phi(\nu_{2})$ for
$Arg(\gamma_{1})>Arg(\gamma_{2})$, where $\nu_{i}\in
H(\mathcal{C})_{\gamma_{i}}$. (here $(\bullet,\bullet)$ is the pairing
between motivic measures and motivic functions.)
\end{prop}

In other words, $\Phi$ can be written as $[\pi: Y\rightarrow\mathcal
{O}b(\mathcal{C})]\mapsto\int_{Y}\mathbb{L}^{\frac{1}{2}\chi(\pi(y),\pi(y))}\widehat{e}_{cl(\pi(y))}$.

\begin{proof}
It suffices to prove the theorem for $\nu_{E_{i}}=[\delta_{E_{i}}:
pt\rightarrow\mathcal {O}b(\mathcal{C})]$, where
$\delta_{E_{i}}(pt)=E_{i}\in\mathcal {O}b(\mathcal{C})$. Recall that
we denote ${\rm dim}\,Ext^i(E,F)$ by $(E,F)_i, i\in \mathbb{Z}$.

We have $\Phi(\nu_{E_{i}})=\mathbb{L}^{\frac{1}{2}\chi(E_{i},
E_{i})}\widehat{e}_{\gamma_{i}}$, which implies that
$\Phi(\nu_{E_{1}})\Phi(\nu_{E_{2}})=\mathbb{L}^{\frac{1}{2}(\chi(E_{1},
E_{1})+\chi(E_{2}, E_{2}))}\widehat{e}_{\gamma_{1}+\gamma_{2}}$.

On the other hand, $\nu_{E_{1}}\cdot\nu_{E_{2}}=\mathbb{L}^{-(E_{2},
E_{1})_{\leq 0}}[\pi_{21}: Ext^{1}(E_{2}, E_{1})\rightarrow\mathcal
{O}b(\mathcal{C})]$. Then
\begin{equation}
\begin{array}{ll}
&\Phi(\nu_{E_{1}}\cdot\nu_{E_{2}})=\mathbb{L}^{-(E_{2}, E_{1})_{\leq
0}}\int_{\alpha\in Ext^{1}(E_{2},
E_{1})}\mathbb{L}^{\frac{1}{2}\chi(E_{\alpha},
E_{\alpha})}\widehat{e}_{\gamma_{1}+\gamma_{2}}\\=&\mathbb{L}^{-(E_{2},
E_{1})_{\leq 0}}\mathbb{L}^{\frac{1}{2}(\chi(E_{1},
E_{1})+\chi(E_{2}, E_{2})+\chi(E_{1}, E_{2})+\chi(E_{2},
E_{1}))}\int_{\alpha\in Ext^{1}(E_{2},
E_{1})}\widehat{e}_{\gamma_{1}+\gamma_{2}}\\=&\mathbb{L}^{-(E_{2},
E_{1})_{\leq0}+\frac{1}{2}(\chi(E_{1},E_{1})+\chi(E_{2},E_{2}))+\chi(E_{2},E_{1})}\mathbb{L}^{(E_{2},E_{1})_{1}}\widehat{e}_{\gamma_{1}+\gamma_{2}}
\\=&\mathbb{L}^{\frac{1}{2}(\chi(E_{1},E_{1})+\chi(E_{2},E_{2}))+(E_{2},E_{1})_{2}}\widehat{e}_{\gamma_{1}+\gamma_{2}}
\end{array}\nonumber
\end{equation}

If $Arg(\gamma_{1})>Arg(\gamma_{2})$, then
$(E_{2},E_{1})_{2}=(E_{1},E_{2})_{0}=0$. Thus
$\Phi(\nu_{E_{1}}\cdot\nu_{E_{2}})=\Phi(\nu_{E_{1}})\Phi(\nu_{E_{2}})$.
\end{proof}

Recall the categories $\CC_V$ and set $V=l$ be a ray. For a generic
central charge $Z$ let us consider the generating function
$$A_{l}^{mot}=\sum_{[E], E\in Ob(\CC_{l})}{{\omega(E)\widehat{e}_{cl(E)}}\over{[Aut(E)]}}=$$
$$\sum_{[E], E\in Ob(\CC_{l})}\mathbb{L}^{\frac{1}{2}(\chi(E,E))}{t^{cl(E)}\over{[Aut(E)]}},$$
where $t=\widehat{e}_{\gamma_0}$ for a primitive $\gamma_0$ such that $Z(\gamma_0)\in l$ generates $Z(\Gamma)\cap l$ and $[Aut(E)]$ denotes the motive of the group of automorphisms of $E$.
More invariantly,  $A_l^{mot}=\Phi(A_l^{Hall})$ where $A_l^{Hall}\in H(\CC_l)$ corresponds to the characteristic function
of the stack of objects of the full subcategory $\CC_l\subset \CC$ generated by semistables $E$ such that $Z(E)\in l$ (cf. loc.cit.).

\begin{defn}
We call $A_l^{mot}$ the motivic DT-series of ${\mathcal C}$ corresponding to the ray $l$.
\end{defn}

Suppose that $\CC$ is associated with the preprojective algebra $\Pi_Q$.
One can show that $A_l^{mot}$ can be obtained from the motivic DT-series for the $3CY$ category associated with $(\widehat{Q},W)$ by the reduction
to $\CC$. Similarly to $A_l^{mot}$ we define $A_V^{mot}$ for any strict sector $V$.

The Proposition 3.6 implies that the series $A_V^{mot}$ is the (clockwise) product of $A_l^{mot}$ over all rays $l\subset V$.
This can be also derived  from the dimensional reduction and the results
of \cite{KoSo2}.

\begin{cor}
The collections of elements $A_{V}^{mot}=\Phi(A_{V}^{Hall})$ parametrized by strict sectors $V\subset \R^2$
with the vertex in the origin
satisfies the Factorization
Property: if a strict sector $V$ is decomposed into a disjoint union
$V=V_{1}\sqcup V_{2}$ in the clockwise order, then
$A_{V}^{mot}=A_{V_{1}}^{mot}A_{V_{2}}^{mot}$.
\end{cor}

\begin{prop}
Motivic DT-series $A_V^{mot}$ is constant on each connected
component of the space of stability conditions.
\end{prop}

\begin{proof}
Similarly to the case of $3CY$ categories, each element $A_V^{mot}$ does not change when
we move in the space of stability conditions on $\CC$ in such a way that central charges of semistable object neither enter nor leave the sector $V$. But in the case of $2CY$ categories the
Euler form is symmetric, hence the motivic quantum torus is commutative. It follows that the wall-crossing formulas from \cite{KoSo2} are trivial. This implies
the result.
\end{proof}

For a $2CY$ category form our class one can construct the corresponding $3CY$ category (see Introduction). We expect that the motivic DT-series arising in this situation are quantum
admissible in the sense of \cite{KoSo3} and can be described in terms of the corresponding COHA (the latter is expected to exist for quite general $3CY$ categories,
see \cite{So}).

Therefore, by analogy with the case of $3CY$ categories,
we can define DT-invariants $\Omega(\gamma)$ in $2CY$ case using (quantum) admissibility (see \cite{KoSo3}, Section 6)  of our DT-series
by the formula:

$$A_V^{mot}=Sym\left(\sum_{n\ge 0}{\mathbb L}^n\sum_{\gamma\ne 0, Z(\gamma)\in V}\Omega(\gamma)\widehat{e}_{\gamma}\right)=$$
$$=Sym\left({\sum_{\gamma\ne 0, Z(\gamma)\in V}\Omega(\gamma)\widehat{e}_{\gamma}\over {1-\mathbb{L}}}\right).$$

By Proposition 3.9 our motivic DT-invariants $\Omega(\gamma)$ depend only on the connected component of $Stab(\CC)$ which contains $Z$. The Conjecture 3.10 (see next subsection) says that $\Omega(\gamma)$ is (essentially) the same as Kac polynomial $a_{\gamma}({\mathbb L})$ (or the motivic DT-invariant of the
corresponding $3CY$ category, see Introduction).

Let us fix the connected component in $Stab(\CC)$ which contains such central charge $Z$  that for each spherical generator $E_i$
of $\CC$ we have $Z(E_i)=(0,...,1,...0)$ (the only nontrivial element $1$ at the $i$-th place).
We will call the corresponding $t$-structure {\it standard}.
We denote the corresponding motivic DT-invariants
by $\Omega_{\CC}^{mot}(\gamma)$.

\subsection{Kac polynomial of a $2CY$ category}

We can now introduce an analog of the Kac polynomial in the case of a $2CY$ category from our class
following the ideas of \cite{Moz1}.

Notice that the coefficient ring
$Mot(Spec(\textbf{k}))[\mathbb{L}^{\frac{1}{2}},\mathbb{L}^{-1},[GL(n)]^{-1}_{n\geqslant
1}]$ of the quantum torus $\mathcal {R}_{\Gamma}$ has a
$\lambda-$ring structure, which can be lifted to the quantum torus
(which is commutative in the case of $2CY$ categories). Recall that
for a $\lambda$-ring we can introduce the operation of
symmetrization by the formula:
$$Sym(r)=\sum_{n\geqslant0}Sym^{n}(r)=\sum_{n\geqslant0}(-1)^{n}\lambda^{n}(-r)=(\sum_{n\geqslant0}(-1)^{n}\lambda^{n}(r))^{-1}.$$
For any ray $l\subset\mathbb{H}_{+}$, where $\mathbb{H}_{+}$ is the
upper half plane, we have the (quantum) admissible element
$A_{l}^{mot}$.

Let ${\mathcal C}$ be a $2CY$ category from our class. We fix the standard  $t$-structure. Recall
the motivic DT-series $A_l^{mot}$.

\begin{conj}

There exist elements $a^{mot}_\gamma(\mathbb L)\in
Mot(Spec(\textbf{k}))[\mathbb{L}^{\frac{1}{2}},\mathbb{L}^{-1},[GL(n)]^{-1}_{n\geqslant
1}]$ which are polynomials in ${\mathbb L}$ and such that the
following formula holds in the (commutative) motivic quantum torus:
$$A_{l}^{mot}=Sym\left(\frac{\sum_{\gamma, Z(\gamma)\in
l}({-} a^{mot}_\gamma(\mathbb L)\cdot{\mathbb L})\widehat{e}_{\gamma}}{1-\mathbb{L}}\right).$$

Furthermore, there exists a $3CY$ category ${\mathcal B}$ such that
the elements $a^{mot}(\mathbb L)$ coincide with motivic
$DT$-invariants with respect to some stability condition on
${\mathcal B}$.

\end{conj}

 Some related results can be found in \cite{CBVdB}, \cite{D2} \cite{HLRV}, and especially in \cite{Moz1}. In fact Theorem 5.1 from \cite{Moz1}
establishes the Conjecture in the framework of quivers. More precisely,
if $\mathcal{C}$ is the $2CY$ category associated with the
preprojective algebra of a  quiver, then for its standard $t$-structure the element
$a^{mot}_\gamma(\mathbb L)$ coincides with the Kac polynomial
$a_{\gamma}(\mathbb{L})$ of the Kac-Moody algebra corresponding to the quiver.

We plan to discuss  the general case in the future work.

\section{Appendix by Ben Davison: Comparison of Schiffmann-Vasserot product with Kontsevich-Soibelman product}

Let $Q$ be a quiver with the set of vertices $Q_0:=I$, and the set
of arrows $Q_1=\Omega$. We can construct the double quiver $\overline{Q}$
and the triple quiver with potential $(\widehat{Q}, W)$ as in
Section 2.2. Recall the representation spaces
$\textbf{M}_{\overline{Q},\gamma}$ and
$\textbf{M}_{\widehat{Q},\gamma}$. Since
$\textbf{M}_{\overline{Q},\gamma}\cong T^\ast(\textbf{M}_{Q,\gamma})$,
there is a natural action of the group
$\textbf{G}_{\gamma}=\prod_{i\in I}GL(\gamma^i,\mathbb{C})$ on the
cotangent space. Let $L$ be the subquiver of $\widehat{Q}$
containing only the loops $l_i$, and denote its representation space
$\textbf{M}_{L,\gamma}$.

Let $\mu_{\gamma}$ denote the moment map
$\textbf{M}_{\overline{Q},\gamma}\rightarrow \mathfrak{gl}_{\gamma}$
or the map
$\textbf{M}_{\widehat{Q},\gamma}\rightarrow\mathfrak{gl}_{\gamma}$
both given by evaluating at $\sum_{a\in Q_1} [a,a^*]$.

In \cite{SV} the following special case was considered. Let
$Q$ be the Jordan quiver (i.e. the quiver with one vertex and one
loop). We denote by $C_n$ the commuting variety, i.e. the variety of
pairs of commuting $n\times n$ matrices.  Then $C_n=\mu_n^{-1}(0)$,
where $\gamma=n$ is the dimension vector.

\subsection{Schiffmann Vasserot product}
Notice that the affine algebraic variety
$\textbf{M}_{\overline{Q},\gamma}$ carries the action of the 2-torus
$T:=(\Cp^*)^2$ given by
\begin{align*}
(z_1,z_2)\cdot\rho(a)=&z_1\rho(a)\\
(z_1,z_2)\cdot\rho(a^*)=&z_2\rho(a^*),
\end{align*}
where $\rho: \mathbb{C}Q\rightarrow Gl(V_{\gamma})$ is a
representation of $\mathbb{C}Q$. Define
$\tilde{\textbf{G}}_{\gamma}:=\textbf{G}_{\gamma}\times T$. For
$\gamma_1,\gamma_2\in\mathbb{N}^{I}$ with $\gamma_1+\gamma_2=\gamma$
we define
$\tilde{\textbf{G}}_{\gamma_1,\gamma_2}=\textbf{G}_{\gamma_1,\gamma_2}\times
T$, and
$\tilde{\textbf{G}}_{\gamma_1\times\gamma_2}=\textbf{G}_{\gamma_1}\times\textbf{G}_{\gamma_2}\times
T$. Following Schiffmann and Vasserot \cite{SV} let us consider the
space
\[
\Ho_{c,\tilde{\textbf{G}}_{\gamma}}(\mu_{\gamma}^{-1}(0))^\vee,
\]
the vector dual to the compactly supported equivariant cohomology
(equivalently, the equivariant Borel-Moore homology).

For $\gamma=\gamma_1+\gamma_2$ denote by
$\gl_{\gamma_1,\gamma_2}\subset \gl_{\gamma}$ the subspace of
parabolic matrices.  Denote by $Z_{\gamma_1,\gamma_2}$ the space of
triples $(a,b,c)\in
\textbf{M}_{\overline{Q},\gamma_1}\times\textbf{M}_{\overline{Q},\gamma_2}\times\gl_{\gamma_1,\gamma_2}$
such that
$c|_{\gl_{\gamma_1}\times\gl_{\gamma_2}}=\mu_{\gamma_1}(a)\times\mu_{\gamma_2}(b)$.

Then we have the following diagram

\[
\xymatrix{
\textbf{M}_{\overline{Q},\gamma}&\textbf{M}_{\overline{Q},\gamma_1,\gamma_2}\ar[l]_h\ar[r]^{f'}&Z_{\gamma_1,\gamma_2}
\\
\mu_{\gamma}^{-1}(0)\ar[u]&\mu_{\gamma_1,\gamma_2}^{-1}(0)\ar[l]^{h|_{\mu^{-1}_{\gamma_1,\gamma_2}(0)}}\ar[u]\ar[r]&\mu_{\gamma_1}^{-1}(0)\times
\mu_{\gamma_2}^{-1}(0)\ar[u]^{g} }
\]
where $f'$ is the map
\[
(\alpha)\mapsto
(\alpha|_{\textbf{M}_{\overline{Q},\gamma_1}},\alpha|_{\textbf{M}_{\overline{Q},\gamma_2}},\mu_{\gamma}(\alpha))
\]
and $g$ is the map
\[
(\alpha,\beta)\mapsto(\alpha,\beta,0)
\]
and we have set
$\mu_{\gamma_1,\gamma_2}^{-1}(0):=\mu_{\gamma}^{-1}(0)\cap\textbf{M}_{\overline{Q},\gamma_1,\gamma_2}$.

\smallbreak The rightmost square is Cartesian, meaning that when we
restrict the natural morphism of sheaves
\[
\left(f'_!\QQ\rightarrow \QQ[\sum_{i\in
I}\gamma_1^i\gamma_2^i-\sum_{a\in \overline{Q}_1}
\gamma_1^{s(a)}\gamma_2^{t(a)}]\right)|_{\mu_{\gamma_1}^{-1}(0)\times
\mu_{\gamma_2}^{-1}(0)}
\]
we obtain a map
\[
f'_*:\Ho_{c,\tilde{\textbf{G}}_{\gamma_1,\gamma_2}}(\mu_{\gamma_1,\gamma_2}^{-1}(0))\rightarrow
\Ho_{c,\tilde{\textbf{G}}_{\gamma_1,\gamma_2}}(\mu_{\gamma_1}^{-1}(0)\times
\mu_{\gamma_2}^{-1}(0)).
\]
We have also the map of sheaves on
$\textbf{M}_{\overline{Q},\gamma}$
\[
\QQ_{\mu_{\gamma}^{-1}(0)}\rightarrow
h_*\QQ_{\mu_{\gamma_1,\gamma_2}^{-1}(0)}
\]
and since $h$ is proper we may replace $h_*$ with $h_!$ and after
taking compactly supported cohomology we obtain a map
\[h^*:\Ho_{c,\tilde{\textbf{G}}_{\gamma_1,\gamma_2}}(\mu_{\gamma}^{-1}(0))\rightarrow \Ho_{c,\tilde{\textbf{G}}_{\gamma_1,\gamma_2}}(\mu_{\gamma_1,\gamma_2}^{-1}(0)).
\]
Via the inclusion $\textbf{G}_{\gamma_1,\gamma_2}\hookrightarrow
\textbf{G}_{\gamma}$ we obtain a map
\begin{equation}
\label{coeffs}
r:\Ho_{c,\tilde{\textbf{G}}_{\gamma}}(\mu_{\gamma}^{-1}(0))\rightarrow
\Ho_{c,\tilde{\textbf{G}}_{\gamma_1,\gamma_2}}(\mu_{\gamma}^{-1}(0))
\end{equation}
as in \cite{KoSo3}. Finally the Schiffmann Vasserot multiplication
is given by setting
\[
m_{SV}=(f'_*h^*r)^{\vee}
\]
where we identify
\[
\Ho_{c,\tilde{\textbf{G}}_{\gamma_1,\gamma_2}}(\mu_{\gamma_1}^{-1}(0)\times
\mu_{\gamma_2}^{-1}(0))^\vee\cong\Ho_{c,\tilde{\textbf{G}}_{\gamma_1}}(\mu_{\gamma_1}^{-1}(0))^\vee\otimes_{\Ho_{T}(\pt)}\Ho_{c,\tilde{\textbf{G}}_{\gamma_2}}(\mu_{\gamma_2}^{-1}(0))^\vee
\]
via the affine fibration
$\textbf{G}_{\gamma_1,\gamma_2}\rightarrow\textbf{G}_{\gamma_1}\times\textbf{G}_{\gamma_2}$
and the $T$-equivariant Kunneth isomorphism.
\bibliographystyle{amsplain}
\subsection{Critical COHA product}
Consider the quiver with potential $(\widehat{Q},
W=\sum_{a\in\Omega}[a, a^{*}]l)$ constructed in Section 2.2. The
torus $T$ acts on $\textbf{M}_{\widehat{Q},\gamma}$ via
\begin{align*}
(z_1,z_2)\cdot\rho(a)=&z_1\rho(a)\\
(z_1,z_2)\cdot\rho(a^*)=&z_2\rho(a^*)\\
(z_1,z_2)\cdot\rho(l_i)=&z_1^{-1}z_2^{-1}\rho(l_i).
\end{align*}
Note that $Tr(W)_{\gamma}$ is invariant with respect to the
$T$-action, and the natural projection
$\textbf{M}_{\widehat{Q},\gamma}\rightarrow\textbf{M}_{\overline{Q},\gamma}$
is $\tilde{\textbf{G}}_{\gamma}$-equivariant. \smallbreak We
consider the space
\[
\bigoplus_{\gamma}\Ho^{crit}_{c,\tilde{\textbf{G}}_{\gamma}}(\textbf{M}_{\widehat{Q},\gamma},W_\gamma)^\vee.
\]
\smallbreak Consider the correspondence diagram
\[
\xymatrix{
\textbf{M}_{\widehat{Q},\gamma}&\ar[l]^{\tau}\textbf{M}_{\widehat{Q},\gamma_1,\gamma_2}\ar[r]_-{\zeta}&\textbf{M}_{\widehat{Q},\gamma_1}\times\textbf{M}_{\widehat{Q},\gamma_2}.
}
\]
Applying $\varphi_{Tr(W)_{\gamma}}$ to the map of sheaves
$\mathbb{Q}\rightarrow\tau_\ast\mathbb{Q}$ we obtain a map
$\varphi_{Tr(W)_{\gamma}}\QQ\rightarrow
\tau_*\varphi_{Tr(W)_{\gamma_1,\gamma_2}}\QQ$ by properness of
$\tau$, and a map
\[
\tau^*:\Ho^{crit}_{c,\tilde{\textbf{G}}_{\gamma_1,\gamma_2}}(\textbf{M}_{\widehat{Q},\gamma},W_{\gamma})\rightarrow
\Ho^{crit}_{c,\tilde{\textbf{G}}_{\gamma_1,\gamma_2}}(\textbf{M}_{\widehat{Q},\gamma_1,\gamma_2},W_{\gamma_1,\gamma_2}).
\]
Since $\zeta$ is an affine fibration, and
$Tr(W)_{\gamma_1,\gamma_2}=(Tr(W)_{\gamma_1}\boxplus
Tr(W)_{\gamma_2})\circ\zeta$, there is a natural shifted isomorphism
\[
\Ho^{crit}_{c,\tilde{\textbf{G}}_{\gamma_1,\gamma_2}}(\textbf{M}_{\widehat{Q},\gamma_1,\gamma_2},W_{\gamma_1,\gamma_2})\rightarrow
\Ho^{crit}_{c,\tilde{\textbf{G}}_{\gamma_1,\gamma_2}}(\textbf{M}_{\widehat{Q},\gamma_1}\times\textbf{M}_{\widehat{Q},\gamma_2},W_{\gamma_1}\boxplus
W_{\gamma_2}).
\]
Changing the group as in (\ref{coeffs}), and composing the duals of
the above two maps, we obtain the $H_T(\pt)$-linear
Kontsevich-Soibelman product $m_{KS}$ (here we use the
Thom-Sebastiani isomorphism rather than Kunneth -- via dimensional
reduction this becomes a special case of the equivariant Kunneth
isomorphism):
\begin{align*}
\Ho^{crit}_{c,\tilde{\textbf{G}}_{\gamma_1}}(\textbf{M}_{\widehat{Q},\gamma_1},W_{\gamma_1})^\vee\otimes_{\Ho_{T}(\pt)}
\Ho^{crit}_{c,\tilde{\textbf{G}}_{\gamma_2}}(\textbf{M}_{\widehat{Q},\gamma_2},W_{\gamma_2})^\vee\xrightarrow{m_{KS,\gamma_1,\gamma_2}}
\Ho^{crit}_{c,\tilde{\textbf{G}}_{\gamma}}(\textbf{M}_{\widehat{Q},\gamma},W_{\gamma})^\vee.
\end{align*}
\subsection{Dimensional reduction}
We give $\textbf{M}_{\widehat{Q},\gamma}$ a further $\Cp^*$-action
as follows
\begin{align*}
z\cdot \rho(a)=&\rho(a)\\
z\cdot \rho(a^*)=&\rho(a^*)\\
z\cdot \rho(l_i)=&z\rho(l_i).
\end{align*}
Then $Tr(W)_{\gamma}$ is an equivariant function with respect to
this action, where $\Cp^*$ acts on the target $\mathbb{C}$ via the
weight one (scaling) action.  There is a projection
$\pi:\textbf{M}_{\widehat{Q},\gamma}\rightarrow
\textbf{M}_{\overline{Q},\gamma}$ given by forgetting the action on
the loops $l_i$.  We write
$\textbf{M}_{\widehat{Q},\gamma}=\textbf{M}_{\overline{Q},\gamma}\times\textbf{M}_{L,\gamma}$.
In this way we consider
$\mu_{\gamma}^{-1}(0)\times\textbf{M}_{L,\gamma}$ as a subspace
of $\textbf{M}_{\widehat{Q},\gamma}$.  Note that
$\mu_{\gamma}^{-1}(0)\times\textbf{M}_{L,\gamma}\subset
Tr(W)_{\gamma}^{-1}(0)$. \smallbreak There is a canonical map
$\varphi_{Tr(W)_{\gamma}}\mathbb{Q}_{\textbf{M}_{\widehat{Q},\gamma}}[-1]\rightarrow
\QQ_{Tr(W)_{\gamma}^{-1}(0)}$.  The dimensional reduction theorem of
the appendix of \cite{D1} (see also \cite{KoSo3} in the case of
rapid decay cohomology) states that the induced map
\[
\pi_!((\varphi_{Tr(W)_{\gamma}}\mathbb{Q}_{\textbf{M}_{\widehat{Q},\gamma}}[-1]\rightarrow
\QQ_{Tr(W)_{\gamma}^{-1}(0)})_{\mu_{\gamma}^{-1}(0)\times\textbf{M}_{L,\gamma}})
\]
is an isomorphism, and so taking compactly supported cohomology we
obtain a (shifted) isomorphism
\begin{equation}
\label{dimred}
\Ho^{crit}_{c,\tilde{\textbf{G}}_{\gamma}}(\textbf{M}_{\widehat{Q},\gamma},W_{\gamma})\rightarrow
\Ho_{c,\tilde{\textbf{G}}_{\gamma}}(\mu_{\gamma}^{-1}(0)).
\end{equation}
This establishes an isomorphism between the underlying vector spaces
of the two COHAs (in fact it is an isomorphism of mixed Hodge
structures, up to the already-mentioned cohomological shift and a
Tate twist).  The purpose of this appendix is to compare the two algebra
structures $m_{SV}$ and $m_{KS}$.
\subsection{Proof of (almost) preservation of product}
The following lemma accounts for the sign that appears in the
comparison of $m_{SV}$ and $m_{KS}$.
\begin{lem}
\label{signLemma} Let
$X=\Spec(\mathbb{C}[x_1,\ldots,x_m])\times\Spec(\mathbb{C}[y_1,\ldots,y_m])\times\Spec(\mathbb{C}[z_1,\ldots,z_n])$
and let
\[
f=\sum_{i\leq m}x_iy_i.
\]
In the following diagram
\[
\xymatrix{
\Ho_c(\pt,\mathbb{Q})&\ar[l]_-a\Ho_{c}(Z(x_1,\ldots,x_m),\mathbb{Q})\\
\Ho_c(Z(y_1,\ldots,y_m),\mathbb{Q})\ar[u]^-{b}&\Ho^{crit}_c(X,f)\ar[u]^c\ar[l]_-d
}
\]
the maps $a$ and $b$ are the canonical shifted isomorphisms associated to
affine fibrations, while the maps $c$ and $d$ are dimensional
reduction isomorphisms.  Then we have the equality
\[
ac=bd\cdot (-1)^m.
\]
\end{lem}
\begin{proof}
By the Thom--Sebastiani isomorphism it is enough to prove the same
result in the case $m=1$ and $n=0$, i.e. consider the maps
\[
\xymatrix{
\Ho_c(\pt,\mathbb{Q})&\ar[l]^{\pi_{x,*}}\Ho_c(Z(x),\mathbb{Q})
\\
\Ho_c(Z(y),\mathbb{Q})\ar[u]^{\pi_{y,*}}&\ar[l]^{(i_y)^*_{\varphi}}\Ho^{crit}_c(\mathbb{A}^2,xy)\ar[u]^{(i_x)^*_{\varphi}}
}
\]
where $i_x:Z(x)\hookrightarrow \mathbb{A}^2$ and
$i_y:Z(y)\hookrightarrow\mathbb{A}^2$ are the obvious inclusions,
and $\pi_x,\pi_y$ are the projections to a point. From now on we adopt the notation
that when we pull back or push forward vanishing cycle cohomology,
we insert a $\varphi$ as a subscript or superscript, as we will be
comparing ordinary compactly supported cohomology with vanishing
cycle cohomology and operations on the two should be distinguished
by the notation.  Now
$\Ho^{crit}_c(\mathbb{A}^2,xy)$ is concentrated in degree 2, and is
explicitly described as follows.  There is a distinguished triangle
\[
\Ho^{crit}_c(\mathbb{A}^2,xy)\rightarrow
\Ho_c(Z(xy),\mathbb{Q})\rightarrow\Ho_c(Z(xy-1),\mathbb{Q}).
\]
The variety $Z(xy)$ has two components, so that
$\Ho_c^2(Z(xy),\mathbb{Q})\cong\mathbb{Q}^{\oplus 2}$, with $(1,0)$
representing the fundamental class of $Z(x)$ and $(0,1)$
representing the fundamental class of $Z(y)$.  On the other hand,
$Z(xy-1)$ has only one component, and we may write the
map
$\Ho_c^2(Z(xy),\mathbb{Q})\rightarrow\Ho^2_c(Z(xy-1),\mathbb{Q})$ as
follows
\[
\mathbb{Q}^{\oplus 2}\xrightarrow{(a,b)\mapsto a+b}\mathbb{Q}
\]
and deduce that $\Ho^{2,crit}_c(\mathbb{A}^2,xy)$ is represented
by the class $\alpha=(1,-1)\in\Ho_c(\{xy=0\})$.  The claim is then
clear; $\alpha$ is sent to
$1\in\Ho_c(\pt,\mathbb{Q})\cong\mathbb{Q}$ by
$\pi_{x,*}(i_x)^*_{\varphi}$ and to $-1$ by
$\pi_{y,*}(i_y)^*_{\varphi}$.

\end{proof}

Throughout the rest of the appendix we will use the following diagram.
\begin{equation}
\label{bigDiag} \xymatrix{
\textbf{M}_{\widehat{Q},\gamma_1}\times\textbf{M}_{\widehat{Q},\gamma_2}&Z_{\gamma_1,\gamma_2}\times \textbf{M}_{L,\gamma_1,\gamma_2}\ar[l]_a\ar[d]_d&\textbf{M}_{\widehat{Q},\gamma_1,\gamma_2}\ar[l]_-b\ar[d]_e\\
\textbf{M}_{\overline{Q},\gamma_1}\times\textbf{M}_{\overline{Q},\gamma_2}\times\textbf{M}_{L,\gamma}\ar[u]^c\ar[r]_-f&Z_{\gamma_1,\gamma_2}\times\textbf{M}_{L,\gamma}&\textbf{M}_{\overline{Q},\gamma_1,\gamma_2}\times\textbf{M}_{L,\gamma}\ar[l]_g\ar[d]_h\\&&\textbf{M}_{\widehat{Q},\gamma}
}
\end{equation}
\begin{rem}
Each of these spaces carries a function induced by $W$, and
regardless of the space we will call this function $Tr(W)$.  For all
the spaces apart from those in the middle column, there is a natural
inclusion of the space inside $\textbf{M}_{\widehat{Q},\gamma}$ and
we just define the function $Tr(W)$ to be the restriction of the
function $Tr(W)$ on the ambient space.  In the case of
$(\rho',\rho'',p)\in Z_{\gamma_1,\gamma_2}$ and
$\xi\in\textbf{M}_{L,\gamma}$ we define
$Tr(W)(\rho',\rho'',p,\xi)=Tr(p\xi)$ to be the function $Tr(W)$ on
$Z_{\gamma_1,\gamma_2}\times\textbf{M}_{L,\gamma}$, and we
restrict this same function to define the function $Tr(W)$ on
$Z_{\gamma_1,\gamma_2}\times\textbf{M}_{L,\gamma_1,\gamma_2}$.
\end{rem}
Next we will define the constituent maps of diagram (\ref{bigDiag}).
\begin{itemize}
\item
We define $a$ as follows.  If $(\rho',\rho'',p)\in
Z_{\gamma_1,\gamma_2}$ and
$\xi\in\textbf{M}_{L,\gamma_1,\gamma_2}$ we forget $p$, and
restrict $\xi$ to the block diagonals to get elements $\xi'$ of
$\textbf{M}_{L,\gamma_1}$ and
$\xi''\in\textbf{M}_{L,\gamma_2}$. Via the natural isomorphism
$\textbf{M}_{\widehat{Q},\gamma}\cong\textbf{M}_{\overline{Q},\gamma}\times\textbf{M}_{L,\gamma}$
we get an element
$(\rho',\xi',\rho'',\xi'')\in\textbf{M}_{\widehat{Q},\gamma_1}\times\textbf{M}_{\widehat{Q},\gamma_2}$.
\item
$b$ is defined via the composition
\[
\textbf{M}_{\widehat{Q},\gamma_1,\gamma_2}\cong\textbf{M}_{\overline{Q},\gamma_1,\gamma_2}\times\textbf{M}_{L,\gamma_1,\gamma_2}\xrightarrow{f'\times
\id}Z_{\gamma_1,\gamma_2}\times\textbf{M}_{L,\gamma_1,\gamma_2}
\]
where $f'$ is as in the definition of $m_{SV}$.
\item
Via forgetting the offdiagonal blocks of a representation there is a
$\textbf{G}_{\gamma_1}\times \textbf{G}_{\gamma_2}$-equivariant
affine fibration $\varpi:\textbf{M}_{L,\gamma}\rightarrow
\textbf{M}_{L,\gamma_1}\times\textbf{M}_{L,\gamma_2}$, and
for $c$ we take the composition
\begin{align*}
&\textbf{M}_{\overline{Q},\gamma_1}\times\textbf{M}_{\overline{Q},\gamma_2}\times
\textbf{M}_{L,\gamma}\\\xrightarrow{\id\times\varpi}
&\textbf{M}_{\overline{Q},\gamma_1}\times\textbf{M}_{\overline{Q},\gamma_2}\times
\textbf{M}_{L,\gamma_1}\times\textbf{M}_{L,\gamma_2}\cong\textbf{M}_{\widehat{Q},\gamma_1}\times\textbf{M}_{\widehat{Q},\gamma_2}.
\end{align*}
\item
$d$ is the product of the identity on $Z_{\gamma_1,\gamma_2}$ with
the obvious inclusion
$i:\textbf{M}_{L,\gamma_1,\gamma_2}\hookrightarrow\textbf{M}_{L,\gamma}$.
\item
$e$ is the composition
\[
\textbf{M}_{\widehat{Q},\gamma_1,\gamma_2}\cong\textbf{M}_{\overline{Q},\gamma_1,\gamma_2}\times\textbf{M}_{L,\gamma_1,\gamma_2}\xrightarrow{\id\times
i}\textbf{M}_{\overline{Q},\gamma_1,\gamma_2}\times\textbf{M}_{L,\gamma}.
\]
\item
$f$ is given by
$f(\rho',\rho'',\xi)=(\rho',\rho'',\mu_{\gamma_1}(\rho')\otimes\mu_{\gamma_2}(\rho''),\xi)$.
\item
$g=f'\times \id$ where $f'$ is as in the definition of $m_{SV}$.
\item
$h$ is the obvious inclusion.
\end{itemize}
\begin{rem}
The functions $Tr(W)$ commute with all of the maps in
(\ref{bigDiag}).
\end{rem}
Denote by $\cdot:\mathbb{Z}^{I}\otimes \mathbb{Z}^{I}\rightarrow
\mathbb{Z}$ the usual dot product
\[
\gamma_1\cdot\gamma_2:=\sum_{i\in I}\gamma_1^{i}\gamma_2^{i}.
\]
\begin{thm}
\label{signTheorem} Let $\nu:
\Ho^{crit}_{c,\tilde{\textbf{G}}_{\gamma}}(\textbf{M}_{\widehat{Q},\gamma},W_\gamma)\rightarrow
\Ho_{c,\tilde{\textbf{G}}_{\gamma}}(\mu_{\gamma}^{-1}(0),\mathbb{Q})$
be the dimensional reduction isomorphism.  Let
$\gamma_1,\gamma_2\in\mathbb{N}^{I}$ with
$\gamma=\gamma_1+\gamma_2$. Then
\[
m_{SV}^{\vee}=(\nu_{\gamma_1}\otimes\nu_{\gamma_2})\circ
m_{KS}^{\vee}\circ\nu^{-1}_{\gamma}\cdot
(-1)^{\gamma_1\cdot\gamma_2}.
\]
\end{thm}
\begin{proof}
Firstly we observe that, with the notation of diagram
(\ref{bigDiag})
\begin{align*}
ab=&\zeta\\
he=&\tau
\end{align*}
where these are the maps appearing in the definition of the KS
product, and so $m_{KS}$ is obtained by the (dual of the operation)
$a^{\varphi}_*b^{\varphi}_*e_{\varphi}^*h_{\varphi}^*=\zeta^{\varphi}_*\tau_{\varphi}^*$
in vanishing cycle cohomology.

The rightmost square of (\ref{bigDiag}) is Cartesian, $d$ is the
inclusion of a vector bundle over $Z_{\gamma_1,\gamma_2}$, and $e$
is the pullback of this inclusion along the map
$\textbf{M}_{\overline{Q},\gamma_1,\gamma_2}\rightarrow
Z_{\gamma_1,\gamma_2}$. Then we obtain commutativity of the square
of sheaves on
$Z_{\gamma_1,\gamma_2}\times\textbf{M}_{L,\gamma}$
\[
\xymatrix{
d_!\QQ&g_!e_!\QQ=d_!b_!\QQ\ar[l]\\
\QQ\ar[u]&g_!\QQ\ar[l]\ar[u]  }
\]
and so the commutativity of the square
\[
\xymatrix{
\Ho^{crit}_{c,\tilde{\textbf{G}}_{\gamma_1,\gamma_2}}(Z_{\gamma_1,\gamma_2}\times \textbf{M}_{L,\gamma_1,\gamma_2},W)&\Ho^{crit}_{c,\tilde{\textbf{G}}_{\gamma_1,\gamma_2}}(\textbf{M}_{\widehat{Q},\gamma_1,\gamma_2},W)\ar[l]^-{b^{\varphi}_*}\\
\Ho^{crit}_{c,\tilde{\textbf{G}}_{\gamma_1,\gamma_2}}(Z_{\gamma_1,\gamma_2}\times
\textbf{M}_{L,\gamma},W)\ar[u]^{d_{\varphi}^*}&\ar[l]^{g^{\varphi}_*}\Ho^{crit}_{c,\tilde{\textbf{G}}_{\gamma_1,\gamma_2}}(\textbf{M}_{\overline{Q},\gamma_1,\gamma_2}\times\textbf{M}_{L,\gamma},W)\ar[u]^{e_{\varphi}^*}.
}
\]
Since the dual of KS product is given by
$a^{\varphi}_*b^{\varphi}_*e^*_{\varphi}h^*_{\varphi}$, we may thus
rewrite it as
\begin{equation}
\label{modKS}
m_{KS}^{\vee}=a_*^{\varphi}d_{\varphi}^*g^{\varphi}_*h_{\varphi}^*.
\end{equation}
The following diagrams of sheaves commute:
\[
\xymatrix{
\varphi_{Tr(W)}\QQ\ar[d]\ar[r]&\varphi_{Tr(W)} h_*\QQ\ar[d]\\
\QQ\ar[r]&h_*\QQ }
\]
on $\textbf{M}_{\widehat{Q},\gamma}$ and
\[
\xymatrix{
\varphi_{Tr(W)}g_!\QQ\ar[r]\ar[d]&\varphi_{Tr(W)}\QQ\ar[d]\\
g_!\QQ\ar[r]&\QQ }
\]
on $Z_{\gamma_1,\gamma_2}\times\textbf{M}_{L,\gamma}$, from
which we deduce the commutativity of the following squares
\[
\xymatrix{
\Ho^{crit}_{c,\tilde{\textbf{G}}_{\gamma_1,\gamma_2}}(\textbf{M}_{\widehat{Q},\gamma},W)\ar[d]^{\cong}\ar[r]^-{h^*_{\varphi}}&\Ho^{crit}_{c,\tilde{\textbf{G}}_{\gamma_1,\gamma_2}}(\textbf{M}_{\overline{Q},\gamma_1,\gamma_2}\times\textbf{M}_{L,\gamma},W)\ar[d]^{\cong}\\
\Ho_{c,\tilde{\textbf{G}}_{\gamma_1,\gamma_2}}(\textbf{M}_{\widehat{Q},\gamma}\cap
\mu^{-1}_{\gamma}(0))\ar[r]^-{h^*}&\Ho_{c,\tilde{\textbf{G}}_{\gamma_1,\gamma_2}}(\left(\textbf{M}_{\overline{Q},\gamma}\times\textbf{M}_{L,\gamma}\right)\cap\mu^{-1}_{\gamma}(0))
}
\]
and
\[
\xymatrix{
\Ho^{crit}_{c,\tilde{\textbf{G}}_{\gamma_1,\gamma_2}}((\textbf{M}_{\overline{Q},\gamma_1,\gamma_2}\times\textbf{M}_{L,\gamma},W)\ar[r]_-{g_*^{\varphi}}\ar[d]^{\cong}&\Ho^{crit}_{c,\tilde{\textbf{G}}_{\gamma_1,\gamma_2}}((Z_{\gamma_1,\gamma_2}\times\textbf{M}_{L,\gamma},W)\ar[d]^{\cong}\\
\Ho_{c,\tilde{\textbf{G}}_{\gamma_1,\gamma_2}}((\textbf{M}_{\overline{Q},\gamma}\times\textbf{M}_{L,\gamma})\cap\mu_{\gamma}^{-1}(0))\ar[r]^{g_*}&\Ho_{c,\tilde{\textbf{G}}_{\gamma_1,\gamma_2}}((Z_{\gamma_1,\gamma_2}\times\textbf{M}_{L,\gamma})\cap\mu_{\gamma}^{-1}(0))
}
\]
where in both cases the vertical isomorphisms are given by
dimensional reduction.

Note that dimensionally reducing along the fibre
$\textbf{M}_{L,\gamma}$ does indeed impose the full moment map
relations, as in each case the derivative of $Tr(W)$ with respect to
the function keeping track of the $i,j$th entry of $l_k$ imposes the
condition that $\mu_{\gamma}(\rho)_{ji}=0$.

Similarly, we have the commuting diagram
\[
\xymatrix{
\Ho_{c,\tilde{\textbf{G}}_{\gamma_1,\gamma_2}}^{crit}(Z_{\gamma_1,\gamma_2}\times\textbf{M}_{L,\gamma},W)\ar[r]^-{f^*_{\varphi}}\ar[d]^{\cong}&\Ho^{crit}_{c,\tilde{\textbf{G}}_{\gamma_1,\gamma_2}}(\textbf{M}_{\overline{Q},\gamma_1}\times\textbf{M}_{\overline{Q},\gamma_2}\times\textbf{M}_{L,\gamma},W)\ar[d]^{\cong}\\
\Ho_{c,\tilde{\textbf{G}}_{\gamma_1,\gamma_2}}((Z_{\gamma_1,\gamma_2}\times\textbf{M}_{L,\gamma})\cap\mu_{\gamma}^{-1}(0))\ar[r]^-{f^*}&\Ho_{c,\tilde{\textbf{G}}_{\gamma_1,\gamma_2}}((\textbf{M}_{\overline{Q},\gamma_1}\times\textbf{M}_{\overline{Q},\gamma_2}\times\textbf{M}_{L,\gamma})\cap\mu_{\gamma}^{-1}(0))
}
\]
where $\mu:Z_{\gamma_1,\gamma_2}\times
\textbf{M}_{L,\gamma}\rightarrow
\mathfrak{g}_{\gamma_1,\gamma_2}$ is defined by
$((\rho',\rho'',p),\xi)\mapsto p$. Putting the previous three
diagrams together, we conclude that the following diagram commutes
\[
\xymatrix{
\Ho^{crit}_{c,\tilde{\textbf{G}}_{\gamma_1,\gamma_2}}(\textbf{M}_{\widehat{Q},\gamma},W)\ar[d]\ar[rr]^-{f^*_{\varphi}g_*^{\varphi}h^*_{\varphi}}&&\Ho^{crit}_{c,\tilde{\textbf{G}}_{\gamma_1,\gamma_2}}(\textbf{M}_{\overline{Q},\gamma_1}\times\textbf{M}_{\overline{Q},\gamma_2}\times\textbf{M}_{L,\gamma},W)\ar[d]\\
\Ho_{c,\tilde{\textbf{G}}_{\gamma_1,\gamma_2}}(\textbf{M}_{\widehat{Q},\gamma}\cap
\mu_{\gamma}^{-1}(0))\ar[rr]^-{m_{SV}^{\vee}\otimes\id_{\Ho_c(\textbf{M}_{L,\gamma})}}&&
\Ho_{c,\tilde{\textbf{G}}_{\gamma_1,\gamma_2}}((\textbf{M}_{\overline{Q},\gamma_1}\times\textbf{M}_{\overline{Q},\gamma_2}\times\textbf{M}_{L,\gamma})\cap\mu_{\gamma}^{-1}(0))
}
\]

Note that the function $Tr(W)$ on
$\textbf{M}_{\overline{Q},\gamma_1}\times\textbf{M}_{\overline{Q},\gamma_2}\times\textbf{M}_{L,\gamma}$
is the box sum of the zero function on the off block-diagonal
entries of $\textbf{M}_{L,\gamma}$ and the function $Tr(W)$ on
$\textbf{M}_{\widehat{Q},\gamma_1}\times\textbf{M}_{\widehat{Q},\gamma_2}$,
so that we have a commuting square
\[
\xymatrix{
\Ho^{crit}_{c,\tilde{\textbf{G}}_{\gamma_1\times\gamma_2}}(\textbf{M}_{\overline{Q},\gamma_1}\times\textbf{M}_{\overline{Q},\gamma_2}\times\textbf{M}_{L,\gamma},W)\ar[r]^-{c^{\varphi}_*}\ar[d]^{\cong^{\textrm{dim
red}}}&\Ho^{crit}_{c,\tilde{\textbf{G}}_{\gamma_1\times\gamma_2}}(\textbf{M}_{\widehat{Q},\gamma_1}\times\textbf{M}_{\widehat{Q},\gamma_2},W)\ar[d]^{\nu^{-1}_{\gamma_1}\otimes\nu^{-1}_{\gamma_2}}
\\
\Ho_{c,\tilde{\textbf{G}}_{\gamma_1\times\gamma_2}}((\textbf{M}_{\overline{Q},\gamma_1}\times\textbf{M}_{\overline{Q},\gamma_2}\times\textbf{M}_{L,\gamma})\cap\mu^{-1}(0))\ar[r]&\Ho_{c,\tilde{\textbf{G}}_{\gamma_1\times\gamma_2}}((\textbf{M}_{\widehat{Q},\gamma_1}\times\textbf{M}_{\widehat{Q},\gamma_2})\cap\mu^{-1}(0))
}
\]
where the bottom (shifted) isomorphism is the usual isomorphism
associated to an affine fibration (note that the function $\mu$ is
insensitive to the $\textbf{M}_{L}$ component).

So putting everything together, we have that
\begin{equation}
\label{finalSV}
m_{SV,\gamma_1,\gamma_2}^{\vee}=(\nu_{\gamma_1}\otimes\nu_{\gamma_2})c_*^{\varphi}f_{\varphi}^*g_*^{\varphi}h^*_{\varphi}\nu^{-1}_{\gamma}
\end{equation}
and
\begin{align}
m_{KS,\gamma_1,\gamma_2}^{\vee}=&a_*^{\varphi}b_*^{\varphi}e^*_{\varphi}h_{\varphi}^*\\=&a_*^{\varphi}d^*_{\varphi}g_*^{\varphi}h_{\varphi}^*&\textrm{using
equation }(\ref{modKS})\label{finalKS}
\end{align}
while what we want to show is that
\begin{equation}
\label{finalComp}
m_{SV,\gamma_1,\gamma_2}^{\vee}=^{?}(\nu_{\gamma_1}\otimes\nu_{\gamma_2})m_{KS,\gamma_1,\gamma_2}^{\vee}\nu^{-1}_{\gamma}
\end{equation}
We will see that we do not actually have a strict equality.  By comparing
(\ref{finalSV}) and (\ref{finalKS}) with the left and right hand
side of (\ref{finalComp}) we see that the left and right hand side
are the same, except we have substituted
$c_*^{\varphi}f^*_{\varphi}$ for $a_*^{\varphi}d^*_{\varphi}$, and
so compatibility of the KS product with the SV product is settled by
(almost) commutativity of the following diagram:
\begin{equation}
\label{twistcomm} \xymatrix{
\Ho^{crit}_{c,\tilde{\textbf{G}}_{\gamma_1\times\gamma_2}}(\textbf{M}_{\widehat{Q},\gamma_1}\times\textbf{M}_{\widehat{Q},\gamma_2})&\Ho^{crit}_{c,\tilde{\textbf{G}}_{\gamma_1\times\gamma_2}}(Z_{\gamma_1,\gamma_2}\times \textbf{M}_{L,\gamma_1,\gamma_2})\ar[l]_{a_*^{\varphi}}\\
\Ho^{crit}_{c,\tilde{\textbf{G}}_{\gamma_1\times\gamma_2}}(\textbf{M}_{\overline{Q},\gamma_1}\times\textbf{M}_{\overline{Q},\gamma_2}\times\textbf{M}_{L,\gamma})\ar[u]^{c_*^{\varphi}}&\ar[l]^-{f^*_{\varphi}}\Ho^{crit}_{c,\tilde{\textbf{G}}_{\gamma_1\times\gamma_2}}(Z_{\gamma_1,\gamma_2}\times\textbf{M}_{L,\gamma})\ar[u]^{d^*_{\varphi}}.
}
\end{equation}
Consider $Z_{\gamma_1,\gamma_2}\times\textbf{M}_{L,\gamma}$ as
a $\tilde{\textbf{G}}_{\gamma_1\times\gamma_2}$-equivariant bundle
over $\textbf{M}_{\widehat{Q},\gamma_1}\times
\textbf{M}_{\widehat{Q},\gamma_2}$. The function $Tr(W)_{\gamma}$
splits as a box sum $\tau\boxplus \zeta$, where $\zeta$ is a
quadratic function on the fibres and $\tau$ is the function $Tr(W)$
on the base. Consider a point $x$ of
$(\textbf{M}_{\widehat{Q},\gamma_1}\times\textbf{M}_{\widehat{Q},\gamma_2},\tilde{\textbf{G}}_{\gamma_1\times\gamma_2})_N$,
and the fibre of
$(Z_{\gamma_1,\gamma_2}\times\textbf{M}_{L,\gamma},\tilde{\textbf{G}}_{\gamma_1\times\gamma_2})_N$
over it.  This fibre splits as a direct sum $A\oplus B\oplus C$,
where
$A\cong\mathfrak{p}_{\gamma_1,\gamma_2}/(\mathfrak{g}_{\gamma_1}\times\mathfrak{g}_{\gamma_2})$,
$B$ is given by the top right hand block of the $\rho(l_i)$, and $C$
is given by the bottom left hand block of $\rho(l_i)$. The factor
$A$ comes from the fibre of the projection
$Z_{\gamma_1,\gamma_2}\rightarrow
\textbf{M}_{\overline{Q},\gamma_1}\times\textbf{M}_{\overline{Q},\gamma_2}$.
Via the Killing form, we have the isomorphism $k:A\cong C^{\vee}$, and $\zeta$
is the function $(a,b,c)\mapsto k(a)(c)$. Applying the
Thom--Sebastiani isomorphism, pushing forward to our point $x$ in
the base $(\textbf{M}_{\widehat{Q},\gamma_1}\times
\textbf{M}_{\widehat{Q},\gamma_2},\tilde{\textbf{G}}_{\gamma_1\times\gamma_2})_N$, the underlying diagram of sheaves of
diagram (\ref{twistcomm}) is given by tensoring the constant map on
$\varphi_{\tau_N}[-1]\mathbb{Q}$ with
\[
\xymatrix{
\Ho_c(\pt,\mathbb{Q})&\ar[l] \Ho_c(A\oplus B,\mathbb{Q})\\
\Ho_c(B\oplus C,\mathbb{Q})\ar[u]&\ar[l]\Ho^{crit}_c(A\oplus B\oplus
C,\zeta)\ar[u]. }
\]

We claim that the diagram commutes, up to the sign $(-1)^{\dim(A)}$.
Explicitly, we may choose $x_1,\ldots,x_n$ coordinates for $A$,
$y_1,\ldots,y_m$ coordinates for $B$, and $z_1,\ldots,z_n$
coordinates for $C$ such that $\zeta=\sum_{i=1}^{n}x_iz_i$.  Note
that after taking cohomology, the bottom map and the rightmost map
are just dimensional reduction isomorphisms, while the other two
maps are the isomorphisms in compactly supported cohomology induced
by affine fibrations. Then the claim is just Lemma \ref{signLemma},
and the result follows.
\end{proof}
Define the modified dimensional reduction map
$\nu_{\gamma}^{\circ}:\Ho^{crit}_{c,\textbf{G}_{\gamma}}(\textbf{M}_{\widehat{Q},\gamma},W)\rightarrow\Ho_{c,\textbf{G}_{\gamma}}(\mu_{\gamma}^{-1}(0),\mathbb{Q})
$ by
\[
\nu_{\gamma}^{\circ}=\nu_{\gamma}\cdot \sqrt{-1}^{\gamma\cdot\gamma}.
\]
\begin{cor}
The map
\[
{\nu}^{\circ\vee}:\bigoplus_{\gamma\in
\mathbb{N}^{I}}\Ho_{c,\textbf{G}_{\gamma}}(\mu_{\gamma}^{-1}(0),\mathbb{Q})^\vee\otimes \mathbb{C}\rightarrow\bigoplus_{\gamma\in
\mathbb{N}^{I}}\Ho^{crit}_{c,\textbf{G}_{\gamma}}(\textbf{M}_{\widehat{Q},\gamma},W_\gamma)^\vee\otimes \mathbb{C}
\]
is an isomorphism of algebras: for all
$\gamma_1,\gamma_2\in\mathbb{N}^{I}$ with $\gamma_1+\gamma_2=\gamma$
there is an equality of maps
\[
m_{KS,\gamma_1,\gamma_2}\circ(\nu_{\gamma_1}^{\circ\vee}\otimes\nu_{\gamma_2}^{\circ\vee})=\nu_{\gamma}^{\circ\vee}\circ
m_{SV,\gamma_1,\gamma_2}.
\]
\end{cor}
\begin{proof}
We have
\begin{align*}
m_{KS,\gamma_1,\gamma_2}\circ(\nu^{\circ\vee}_{\gamma_1}\otimes\nu_{\gamma_2}^{\circ\vee})=&\sqrt{-1}^{\gamma_1\cdot\gamma_1+\gamma_2\cdot\gamma_2}m_{KS,\gamma_1,\gamma_2}\circ(\nu_{\gamma_1}^{\vee}\otimes\nu_{\gamma_2}^{\vee})\\
=&\sqrt{-1}^{\gamma_1\cdot\gamma_1+\gamma_2\cdot\gamma_2}\nu^{\vee}_{\gamma}\circ m_{SV,\gamma_1,\gamma_2}\cdot (-1)^{\gamma_1\cdot\gamma_2}&(\textrm{Theorem }\ref{signTheorem})\\
=&\nu^{\vee}_{\gamma}\circ m_{SV,\gamma_1,\gamma_2}\cdot\sqrt{-1}^{\gamma\cdot\gamma}\\
=&\nu^{\circ\vee}_{\gamma}\circ m_{SV,\gamma_1,\gamma_2}.
\end{align*}
\end{proof}

As we can see from the proof, the isomorphism holds over the subfield $\mathbb{Q}[i]\subset \mathbb{C}$, where $i=\sqrt{-1}$.

Addresses:

Y.S.: Department of Mathematics, KSU, Manhattan, KS 66506, USA, {soibel@math.ksu.edu}

J.R.: Department of Mathematics, KSU, Manhattan, KS 66506, USA, {jren@ksu.edu}

\end{document}